\documentclass[11pt]{article}
\usepackage{caption}
\usepackage{subcaption}
\usepackage{cite}
\usepackage{graphics}
\usepackage{authblk}
\usepackage{amsmath,amssymb,amsfonts}
\usepackage[english]{babel}
\usepackage{hyperref}
\usepackage{array}
\usepackage[rightcaption]{sidecap}
\usepackage{longtable}
\usepackage{enumitem}
\usepackage{scalerel}
\usepackage{amsthm}

\usepackage{graphicx}
\usepackage{color}
\usepackage{algorithm}
\usepackage{algorithmic}
\usepackage{comment}
\numberwithin{equation}{section}
\definecolor{darkgreen}{rgb}{0,0.7,0}

\newcommand{\RR}{\mathbb{R}}

\newcommand{\norm}[1]{\Vert #1 \Vert}

\newcolumntype{M}[1]{>{\centering\arraybackslash}m{#1}}

\usepackage[english]{babel}
\providecommand{\keywords}[1]{\textbf{{Keywords:}} #1}
\numberwithin{equation}{section}
\usepackage{color}

	\newtheorem{remark}{Remark}[section]
	\newtheorem{fig}{Figure}[section]
	
	\newtheorem{definition}{Definition}[section]
	\newtheorem{theorem}{Theorem}[section]
	
	\newtheorem{conj}{Conjecture}[section]
	\newtheorem{lemma}{Lemma}[section]
	\newtheorem{proposition}{Proposition}[section]

\title{Global stabilization of a Sterile Insect Technique model\\ by feedback laws}
\author{Kala AGBO BIDI}
\author{Lu\'\i s ALMEIDA}
\author{Jean-Michel CORON}
\affil{Sorbonne Universit\'{e}, Universit\'{e} Paris Cit\'{e}, CNRS, INRIA,  Laboratoire Jacques-Louis Lions, LJLL, F-75005 Paris, France,
\texttt{kala.agbo\_bidi@sorbonne-universite.fr},  \texttt{luis.almeida@cnrs.fr}, \texttt{jean-michel.coron@sorbonne-universite.fr}} 
\date{\empty}

\begin{document}

	\maketitle
	\begin{flushright}
		\small
		\textit{To the memory of Andrea Bacciotti, a wonderful person,\\ a leader in the field of control theory.}
	\end{flushright}
	\begin{abstract}
		This work concerns feedback global stabilization of the sterile insect technique dynamics. The Sterile Insect Technique (SIT) is presently  one of the most ecological methods for  controlling insect pests responsible for crop destruction and disease transmission worldwide.
		
This technique consists in releasing sterile males among the insect pest population, the aim being to reduce fertility and, consequently, reduce significantly the wild insect population after a few generations.
		
		In this work, we study the global stabilization of a pest population at extinction equilibrium by the SIT method and construct explicit feedback laws that stabilize the model. Numerical simulations show the efficiency of our feedback laws.
	\end{abstract}

\noindent \keywords{Sterile Insect Technique, Pest control,  Dynamical control system, Feedback design, Backstepping feedback, Lyapunov global stabilization, Mosquito population control, Vector borne disease.}

	\section{Introduction}
	Mosquitoes are known to transmit a variety of diseases such as malaria, dengue, yellow fever, Zika virus and many others. These diseases are responsible for a significant number of deaths worldwide.
	According to the World Health Organization (WHO),
	the number of malaria cases worldwide in 2022 was
	estimated at 249 million in 85 endemic countries and territories,
	 an increase of 5 million compared to 2021.
	 The estimated number of deaths in
	 2022 is 608,000 (see \cite{world2023world}).
	 Dengue and Zika virus, also transmitted by mosquitoes,
	 are estimated to cause hundreds of thousands of cases
	 and thousands of deaths each year.
	 In 2023, there was an unexpected increase
	 in dengue cases, resulting in an all-time high
	 of more than five million cases and more than
	 5,000 dengue-related deaths reported in more than
	 80 countries/territories and five WHO regions: Africa,
	 the Americas, South-East Asia, the Western Pacific, and the Eastern Mediterranean (see \cite{world2023Dengue}).

	Although there are many effective vector control measures for malaria and arboviroses, some of them can have negative impact on the environment and may result in ecological damage. For example,  insecticide spraying can have unintended effects on non-target organisms, including beneficial insects such as bees and butterflies \cite{sanchez2021indirect,kaur2024pesticides}. In addition, repeated use of insecticides often leads to the development of resistance in mosquito populations \cite{parakrama2018}.
	
The sterile insect technique (SIT) has been proposed as an alternative tool for reducing mosquito populations. The technique involves sterilizing male mosquitoes (frequently this is done using ionizing radiation) and then releasing them into the wild to mate with wild females.  This strategy was  initially applied successfully (since the 1950s) to nearly eradicate the screw-worm fly in
	North America.  Since then, this technique has also been used for different agricultural pests and disease vectors \cite{barclay1980sterile, bourtzis2021sterile, vreysen2006sterile}.
	
	One advantage of using such a technique is that it only targets the desired species and also significantly reduces the impact on the ecosystem. This is why this technique is increasingly used for the control of insect pests and insect disease vectors. Some previous works have considered applications of feedback controls to SIT: impulsive feedback controls for a 3-D  model
\cite{2019-Bliman-Cardona-Salgado-Dumont-Vasilieva-MB, 2022-Bliman-Dumont-MB},  optimal  controls for a 2-D model \cite{almeida2022optimal} and even optimal impulsive controls for an epidemic model for a vector borne disease in the human population \cite{almeida2024outbreak}.

 For the sake of simplicity, in this paper we chose to focus our presentation on the particular and important case of mosquito population control but many of the results presented can be extended to the use of SIT for the control of other pests.
	
	In order to determine the appropriate releases of sterile males to approach the extinction equilibrium of the population, we use mathematical control theory which provides the necessary tools for constructing such a control.
	Our work involves building this feedback law starting from the model proposed in \cite{strugarek2019use} without the Allee effect.
	Our theoretical results are illustrated with numerical simulations. Moreover, in  section \ref{see:Comparativesect} we do a comparative study between the different feedback laws.

\begin{remark}
While we were finishing writing this work, we learned that the reduced system (system of two ODE studied in \cite{almeida2022optimal}) was also recently studied by A. Cristofaro and L. Rossi  in \cite{2023Rossi}. In particular, they were able to construct a feedback law leading to global stabilization of the extinction equilibrium in this setting using a backstepping approach.
\end{remark}

\section{Mathematical modeling of mosquito population dynamics}
	
\subsection{Mathematical modeling of wild mosquito population dynamics }
The life cycle of mosquitoes has many stages but we will consider a simplified model where we just separate an aquatic and an adult phase. The aquatic phase includes egg, larva  and pupa stages. After the pupa stage, adult mosquitoes emerge and it is in the adult phase that mosquitoes reproduce and only female mosquitoes bite.

As a matter of fact, in order to lay their eggs, female mosquitoes need not only to be fertilized by males but also to have a blood meal. Thus, every 4-5
days, they will take a blood meal (that can sometimes involve biting several victims) and lay 100 to 150 eggs in different places (10 to 15 per place). An adult mosquito usually lives for 2 to 4 weeks. The mathematical model we present takes into account the two phases: the aquatic phase that we denote by $E$ and the  adult phase that we split  into two
sub-compartments, males, $M$,  and females, $F$. We consider the dynamics presented in \cite{strugarek2019use}. Based on this model and neglecting the Allee effect (i.e.  taking  $\beta =+\infty$ in system (2) of \cite{strugarek2019use}, which is the less favorable case for stabilizing the zero solution), we obtain the system

\begin{align}
	&\dot{E} = \beta_E F \left(1-\frac{E}{K}\right) - \big( \nu_E + \delta_E \big) E,\label{eq:S11E1}  \\
	&\dot{M} = (1-\nu)\nu_E E - \delta_M M, \label{eq:S11E2} \\
	&\dot{F} =\nu\nu_E E -  \delta_F F, \label{eq:S11E3}
\end{align}
where,
\begin{itemize}
	\item $E(t)\geq 0$ is the mosquito density in aquatic phase at time $t$;
	\item $M(t)\geq 0$ is the wild  adult  male density at time $t$;
	\item $F(t)\geq 0$ is the density of adult females  at time $t$; we have supposed that all females are immediately fertilized in this setting and this equation is the only one that changes when we add the sterile males in which case only a fraction of the females will be fertilized;
	\item $\beta_E>0$ is the oviposition rate;
	\item $\delta_E,\delta_M,\delta_F >0 $ are the death rates for eggs, wild adult males and fertilized females respectively;
	\item $\nu_E>0$ is the hatching rate for eggs;
	\item $\nu\in (0,1)$ the probability that a pupa gives rise to a female, and $(1-\nu)$ is, therefore, the probability that it gives rise to a male.  And, to simplify, we suppose  females  become fertilized immediately when they emerge from the pupal stage;
	\item $K>0$ is the environmental capacity for eggs. It can be interpreted as the maximum density of eggs that females can lay in breeding sites. Since here the larval and pupal compartments are not present, it is as if $E$ represents all the aquatic compartments, in which case in this term $K$ represents a logistic law's carrying capacity for the aquatic phase that also includes the effects of competition between larvae.  It has the dimensions of a spatial density.
\end{itemize}

We set $x=(E, M, F)^T$ and $\mathcal{D} = \RR^3_+ = \{x\in \RR^3: x\geq 0\}$. The model \eqref{eq:S11E1}-\eqref{eq:S11E3} can be written in the  form
\begin{align}
	\dot{x} = f(x)\label{eq:origine},
\end{align}
where $f:\RR^3\to\RR^3$ represents the right hand side of  \eqref{eq:S11E1}-\eqref{eq:S11E3}. The map $f$ is  continuously differentiable on $\RR^3$. Note that if $\dot x =f(x)$ and $x(0)\in \mathcal{D}$, then, for every $t\geq 0$, $x(t)$ is defined and belongs to $\mathcal{D}$.
Setting the right hand side of \eqref{eq:S11E1}-\eqref{eq:S11E3} to zero we obtain the extinction equilibrium $\textbf{0} = (0, 0, 0)^T$ and the non-trivial equilibrium $x^*=(E^*,M^*, F^*)^T$ given by
\begin{align}
	&{E}^* = K(1 -\frac{1}{\mathcal{R}_0}),\label{eq:equilibreE-1}  \\ &{M}^* = \frac{(1-\nu)\nu_E}{\delta_M}{E}^*\label{eq:equilibreM-1},\\&{F}^* = \frac{\nu\nu_E}{\delta_F}{E}^*,
\end{align}
where
\begin{align}
\mathcal{R}_0 := \frac{\beta_E\nu\nu_E}{\delta_F(\nu_E+\delta_E)}.\label{eq:hypotheseR0}
\end{align}
Note that $x^*\in \mathcal{D}$ if and only if $\mathcal{R}_0\geq 1$.
Let us now recall some definitions connected to the stability of an equilibrium.
\begin{definition}\label{def-stab}
	Let  $x_e\in\mathcal{D}$ be an equilibrium (of \eqref{eq:origine}). The equilibrium $x_e$ is stable in $\mathcal{D}$ if, for every $\varepsilon>0$, there exists a $\delta>0$ such that
	\begin{align}\label{def-stability}
		\left(x_0\in \mathcal{D} \text{ and } \norm{x(0)-x_e}<\delta \right)\Longrightarrow \left(\norm{x(t)-x_e}<\varepsilon \;\;\mbox{for all}\;\; t>0\right).
	\end{align}
The equilibrium $x_e$ is unstable in $\mathcal{D}$ if it is
not stable  in $\mathcal{D}$. It is an attractor in $\mathcal{D}$ if there exists $\eta>0$ such that,  for every initial data $x(0)$ in $\mathcal{D}$ satisfying $\norm{x(0)-x_e}<\eta$, $x(t)\to x_e$ as $t\to \infty$. It is a global attractor  in $\mathcal{D}$ if, for every initial data in $\mathcal{D}$, $x(t)\to x_e$ as $t\to \infty$. It is locally asymptotically stable in $\mathcal{D}$ if it is both stable and an attractor in $\mathcal{D}$. Finally, it is globally asymptotically stable  in $\mathcal{D}$ if it is both stable and  a global attractor  in $\mathcal{D}$.
\end{definition}

The Jacobian of system \eqref{eq:S11E1}-\eqref{eq:S11E3} computed at the extinction equilibrium is
\begin{align}
	J(\textbf{0})=\begin{pmatrix}
		-(\nu_E+\delta_E) &0&\beta_E\\
		(1-\nu)\nu_E&-\delta_M&0\\
		\nu\nu_E &0 & -\delta_F
	\end{pmatrix}\label{eq:jacobian}.
\end{align}
Its characteristic polynomial is
\begin{multline}
	P(\lambda)=\lambda^3 + (\nu_E+\delta_E +\delta_M+\delta_F)\lambda^2\\ + ((\nu_E+\delta_E)\delta_F-\beta_E\nu\nu_E  + \delta_M(\nu_E+\delta_E))\lambda +\delta_M((\nu_E+\delta_E)\delta_F-\beta_E\nu\nu_E).
\end{multline}
Its roots are $-\delta_M$ and the roots of equation
\begin{align}
	\lambda^2 + (\nu_E+\delta_E +\delta_F)\lambda + \delta_F(\nu_E+\delta_E)(1-\mathcal{R}_0) = 0
\end{align}
If $\mathcal{R}_0<1$, all eigenvalues of $J(\textbf{0})$ are either negative or have negative real parts, which implies that  $\textbf{0}$ is locally asymptotically stable. If $\mathcal{R}_0=1$ the eigenvalues of $J(\textbf{0})$  are $-\delta_M$, $0$, and $-\left(\nu_E+\delta_E+\delta_F\right)<0$. If $\mathcal{R}_0>1$,  the eigenvalues of $J(\textbf{0})$ are all real, one is strictly positive, two are strictly negative.

The global stability properties of the extinction equilibrium $\textbf{0} = (0, 0, 0)^T$ are described in terms of the basic offspring number $\mathcal{R}_0$ of the population. This is a key parameter in the theory of population dynamics. Depending on its value, more precisely if and only if $\mathcal{R}_0> 1$, there exists a non-trivial equilibrium point (see \cite{yang2009assessing,almeida2022optimal}).
The essential properties of the model \eqref{eq:S11E1}-\eqref{eq:S11E3}  are summarized in the following theorem similar to \cite[Theorem 7]{anguelov2012mathematical} and \cite[Theorem 1]{anguelov2020sustainable}.
\begin{theorem}\label{th-as-without-Ms} The following properties hold.
	\begin{enumerate}[label=\textbf{(P.\arabic*)}]
\item\label{ithR0} If  $\mathcal{R}_0\leq 1$,  then   $\textbf{0}\in \RR^3$  is a globally asymptotically  stable equilibrium in $\mathcal{D}$ for \eqref{eq:origine};
\item \label{iithR0}  If  $\mathcal{R}_0>1$, then the system has two equilibria $\textbf{0}$ and $x^*$ in $\mathcal{D}$, where $x^*$ is stable with basin of attraction $\mathcal{D}\setminus\{x=(E, M, F)^T\in \RR^3_+ : E=F=0\}$ and $\textbf{0}$ is unstable in $\mathcal{D}$ with the non negative $M-axis$ being a stable manifold.
	\end{enumerate}
\end{theorem}
\begin{proof}
Let us first prove \ref{ithR0}. We could proceed as in the proof of \cite[Theorem 7 (i)]{anguelov2012mathematical} or  \cite[1) in Theorem 1]{anguelov2020sustainable} which are based on properties of monotone operators. We propose a different approach, now based on Lyapunov functions. Let $t\mapsto x(t)=(E(t),M(t),F(t))^T$ be a solution of \eqref{eq:origine} defined at time $0$ and such that $(E(0),M(0),F(0))^T\in \mathcal{D}$. One has
\begin{gather}
\label{value-M}
M(t)=e^{-\delta_Mt}M(0)+(1-\nu)\nu_E\int_0^te^{-\delta_M(t-s)}E(s)\;ds,
\end{gather}
which implies that
\begin{gather}
\label{value-M-stable}
M(t)\leq M(0)+ \frac{(1-\nu)\nu_E}{\delta_M}\sup\{E(s);\; s\geq 0\},
\end{gather}
\begin{multline}
M(t)\leq M(0)e^{-\delta_Mt}+ \frac{(1-\nu)\nu_E}{\delta_M}e^{-\delta_Mt/2} \max\{E(s);\; s\in[0, t/2]\} \\
+ \frac{(1-\nu)\nu_E}{\delta_M}\sup\{E(s);\; s\geq t/2\}.
\label{value-M-attractor}
\end{multline}
Inequality  \eqref{value-M-stable} shows that $\textbf{0}\in \RR^3$  is a  stable equilibrium in $\mathcal{D}$ for \eqref{eq:origine} if $\textbf{0}\in \RR^2$  is a  stable equilibrium in $[0,+\infty)^2$ for the subsystem in $(E,F)^T\in[0,+\infty)^2$:
\begin{align}
		&\dot{E} = \beta_E F \left(1-\frac{E}{K}\right) - \left( \nu_E + \delta_E \right) E,\label{eq:S211E1}  \\
		&\dot{F} =\nu\nu_E E -  \delta_F F. \label{eq:S211E3}
\end{align}
Inequality \eqref{value-M-attractor} shows that $\textbf{0}\in \RR^3$  is a  global attractor in $\mathcal{D}$ for \eqref{eq:origine} if $\textbf{0}\in \RR^2$  is a  global attractor in $[0,+\infty)^2$ for the subsystem \eqref{eq:S211E1}-\eqref{eq:S211E3} in $(E,F)^T\in[0,+\infty)^2$.

Hence, in order to prove \ref{ithR0}, it suffices to check that $\textbf{0}\in [0,+\infty)^2$ is globally asymptotically stable in $[0,+\infty)^2$ for the system \eqref{eq:S211E1}-\eqref{eq:S211E3}. To prove this last statement, let us consider the Lyapunov function $V:[0,+\infty)^2\rightarrow \RR$, $y=(E,F)^T\mapsto V(y)$,  defined by
\begin{gather}
V(y):=\delta_F E+ \beta_E F.
\end{gather}
Then,
\begin{gather}
\text{ V is of class $\mathcal{C}^1$},
\\
\label{EFV>0}
V(y)>V((0,0)^T)=0,\;  \forall y \in [0, +\infty)^2\setminus\{(0,0)^T\},
\\
\label{EFVinftyatinfty}
\text{$V(y)\to +\infty$  when $\norm{y}\to +\infty$ with $y \in [0,+\infty)^2$}.
\end{gather}
The time-derivative of $V$ along the trajectories of \eqref{eq:S211E1}-\eqref{eq:S211E3}  is
\begin{gather}
\label{EVdotV}
\dot V = - \left( \delta_F\left(\nu_E + \delta_E\right)-\beta_E\nu\nu_E \right) E -\frac{\delta_F\beta_E}{K}EF.
\end{gather}
Let us now assume that
\begin{gather}
\label{R0leq1}
\mathcal{R}_0\leq 1.
\end{gather}
From \eqref{EVdotV} and \eqref{R0leq1} one gets
\begin{gather}
\label{EVdotVleq}
\dot V \leq -\frac{\delta_F\beta_E}{K}EF\leq 0.
\end{gather}
We are going to conclude by using the LaSalle invariance principle. Let us assume that we have a trajectory $t\in \RR \mapsto y(t)=(E(t),F(t))^T\in[0,+\infty)^2$ of \eqref{eq:S211E1}-\eqref{eq:S211E3} such that
\begin{gather}
\label{V=0forallt}
\dot V(y(t))=0 \; \forall t\in \RR.
\end{gather}
Then, using \eqref{EVdotVleq},
\begin{gather}
\label{EF=0forallt}
E(t)F(t)=0 \; \forall t\in \RR.
\end{gather}
Let us assume that there exists $t_0\in\RR$ such that
\begin{gather}
\label{Et0not0}
E(t_0)\not =0.
\end{gather}
Then there exists $\varepsilon>0$ such that
\begin{gather}
\label{Etnot0}
E(t)\not =0 \; \forall t \in (t_0-\varepsilon, t_0+\varepsilon),
\end{gather}
which, together with \eqref{EF=0forallt}, implies that
\begin{gather}
\label{F(t)=0}
F(t) =0 \; \forall t \in (t_0-\varepsilon, t_0+\varepsilon).
\end{gather}
Differentiating \eqref{F(t)=0} with respect to time and using \eqref{eq:S211E3} we get
\begin{gather}
\label{E(t)=0}
E(t) =0 \; \forall t \in (t_0-\varepsilon, t_0+\varepsilon),
\end{gather}
in contradiction with \eqref{Etnot0}. Hence
\begin{gather}
\label{E=0-all-time}
E(t) =0 \; \forall t \in \RR.
\end{gather}
Differentiating \eqref{E=0-all-time} with respect to time and using \eqref{eq:S211E1} we get that
\begin{gather}
\label{F(t)=0-new}
F(t) =0 \; \forall t \in \RR.
\end{gather}
With the LaSalle invariance principle, this concludes the proof of \ref{ithR0}.
\begin{remark}
In the case where $\mathcal{R}_0<1$  a simple linear strict Lyapunov function for the full system \eqref{eq:origine} is given in Remark~\ref{rem-strict-Lyapunov}.
\end{remark}

Let us now prove \ref{iithR0}. We first note that one has the following lemma, whose proof is obvious  and is omitted.
\begin{lemma} Let $t\mapsto x(t)=(E(t),M(t),F(t))^T$ be a solution of \eqref{eq:origine} defined at time $0$ and such that $(E(0),M(0),F(0))^T\in \mathcal{D}$. Then it is defined on $[0,+\infty)$. Moreover, if $E(0)\geq K$, then there exists one and only one time $t_0\geq 0$ such that $E(t_0)=K$ and one has
\begin{gather}
E(t)<K\; \forall t> t_0.
\end{gather}
\end{lemma}

Thanks to this lemma we are allowed to assume that $E<K$, which we do from now on. We then follow the proof of \cite[Theorem 7 (ii)]{anguelov2012mathematical}. To prove the stability and basin of attraction of the non-trivial equilibrium $x^*$ we use \cite[Theorem 2.2 in Chapter 2]{smith1995monotone}. This theorem  applies to strongly monotone systems. The Jacobian \eqref{eq:jacobian} associated with \eqref{eq:origine} is not irreducible. Let us consider the subsystem for $E$ and $F$, that is \eqref{eq:S211E1}-\eqref{eq:S211E3}, which defines a dynamical system on $\RR_+^2$. Its Jacobian
		\begin{align}
			j((E,F)^T)=\begin{pmatrix}
				-(\nu_E+\delta_E)-\frac{\beta_E F}{K} & &\beta_E(1-\frac{E}{K})\\
				\nu\nu_E  & & -\delta_F
			\end{pmatrix}\label{eq:jacobian2}
		\end{align}
		is irreducible.  Considering the usual coordinate-wise comparison and applying  \cite[Theorem 2.2 in Chapter 2]{smith1995monotone} to the two dimensional interval
		\begin{align}
			\{(E,F)^T\in\RR_+^2: 0\leq E\leq E^*,0\leq F\leq F^*\},
		\end{align}
		it follows that every solution starting in this interval, excluding the end points $(0,0)^T$ and $(E^*,F^*)^T$, converges to one of the end points. 
		
		The Jacobian at $\textbf{0}=(0,0)^T$ is
		\begin{align}
			j(\textbf{0})=\begin{pmatrix}
				-(\nu_E+\delta_E) & &\beta_E\\
				\nu\nu_E  & & -\delta_F
			\end{pmatrix}\label{eq:jacobian00}
		\end{align}
Its   characteristic equation is
		\begin{align}
			\lambda^2 + (\delta_F+\nu_E+\delta_E)\lambda + \delta_F(\nu_E+\delta_E)-\beta_E\nu\nu_E=0,
		\end{align}
		whose discriminant is
		\begin{align}
			\Delta = (\nu_E+\delta_E-\delta_F)^2+4\beta_E\nu\nu_E\geq 0.
		\end{align}
The eigenvalues are 
\begin{gather}
	\lambda_-:= -\frac{(\delta_F+\nu_E+\delta_E) +\sqrt{\Delta}}{2}\\
	\lambda_+:= \frac{-(\delta_F+\nu_E+\delta_E) +\sqrt{\Delta}}{2}
\end{gather}
Therefore,  since $\mathcal{R}_0>1$,  $\lambda_+ >0$ and  so $\textbf{0}$ is unstable. Since $j(\textbf{0})$ is a Metzler matrix, it has a strictly positive eigenvector corresponding to the positive eigenvalue  $\lambda_+ >0$, which is 
\begin{gather}
	v_+ = \begin{pmatrix}
		1\\
		 \frac{(\nu_E+\delta_E-\delta_F) +\sqrt{\Delta}}{2\beta_E}
	\end{pmatrix}
\end{gather}
Moreover, the eigenvector corresponding to the negative eigenvalue $\lambda_-$ is
    \begin{gather}
    v_- = \begin{pmatrix}
        1\\
         \frac{(\nu_E+\delta_E-\delta_F) -\sqrt{\Delta}}{2\beta_E}
    \end{pmatrix}
\end{gather}
which has two components with opposite signs and is thus biologically meaningless.  Hence, no solution converges to $\textbf{0}=(0,0)^T$ except the trivial solution which is identically equal to $\textbf{0}=(0,0)^T$. Therefore, every nontrivial solution converges to $(E^*,F^*)^T$. The implication for the three dimensional system \eqref{eq:S11E1}-\eqref{eq:S11E3} is that all solutions starting in the interval $[\textbf{0},x^*],$ excluding  the $M$-axis, converge to $x^*=(E^*,M^*,F^*)^T$.

		Using the same argument as in \cite{anguelov2020sustainable}, any solution starting at a point larger than $x^*$ converges to $x^*$. Since any point in  $ \mathcal{D}\setminus\{x=(E, M, F)^T\in \RR^3_+ : E=F=0\}$ can be placed between a point below $x^*$, but not on the $M$-axis, and a point above $x^*$, every solution starting in $ \mathcal{D}\setminus\{x=(E, M, F)^T\in \RR^3_+ : E=F=0\}$  converges to $x^*$. The monotone convergence of the solutions initiated below and above $x^*$ implies the stability of $x^*$ as well. This concludes the proof of \ref{iithR0} and of Theorem~\ref{th-as-without-Ms}.
	\end{proof}
	\subsection{SIT model in mosquito population dynamics }
The SIT model obtained neglecting the Allee effect from the one presented in \cite{strugarek2019use} is
	\begin{align}
		&\dot{E} = \beta_E F \left(1-\frac{E}{K}\right) - \big( \nu_E + \delta_E \big) E,\label{eq:S1E1}  \\
		&\dot{M} = (1-\nu)\nu_E E - \delta_M M, \label{eq:S1E2} \\
		&\dot{F} =\nu\nu_E E \frac{M}{M+\gamma_s M_s} - \delta_F F, \label{eq:S1E3} \\
		&\dot{M}_s = u - \delta_s M_s, \label{eq:S1E4}
	\end{align}
	where $M_s(t)\geq 0$ is the density of sterilized adult males,  $\delta_s>0$ is the death rate of sterilized adults, $u\geq 0$ is the control which is the density of sterile males released at time $t$, and $0<\gamma_s\leq1$ accounts for the fact that females may have a preference for fertile males. Then, the probability that a female mates with a fertile male is ${M}/{(M + \gamma_sM_s)}$. From now on we assume that
\begin{gather}
\label{deltas>deltaM}
\delta_s\geq \delta_M,
\end{gather}
which is a biologically relevant assumption (and even if this were not so, the sterile males would have a competitive advantage due to a higher longevity that would make SIT more efficient).

Let $\mathcal{D}' := [0,+\infty)^4$. When applying a feedback law $u:\mathcal{D}'\rightarrow [0,+\infty)$, the closed-loop system is  the  system
\begin{equation}
	\dot{x}= H(x,u(x)),
			\label{eq:closed-loop}
\end{equation}
		where
		\begin{equation}
			H(x,u) = \left( \begin{array}{ccc} \beta_E F \left(1-\frac{E}{K}\right) - \big( \nu_E + \delta_E \big) E\\  (1-\nu)\nu_E E - \delta_M M\\  \nu\nu_E E \frac{M}{M+\gamma_s M_s} - \delta_F F\\ u -\delta_sM_s
			\end{array}
			\right).	
		\end{equation}
Concerning the regularity of the feedback law, we always assume that
\begin{gather}
\label{reg-u_linfty}
u\in L^\infty_{\text{loc}}(\mathcal{D}').
\end{gather}
Note that, even if $u$ is of class $\mathcal{C}^\infty$, the map $x\in \mathcal{D}'\mapsto H(x,u(x))\in \RR^4$ is not continuous and one needs to specify the definition of the solutions for the closed-loop system \eqref{eq:closed-loop}. Carath\'{e}odory solutions  seem to be natural candidates. Roughly speaking, Carath\'{e}odory
solutions are absolutely continuous curves that satisfy the integral version of the differential equation. These solutions are indeed useful in other contexts. However, if they can lead to robustness for small errors on the control, as shown in \cite{2002-Ancona-Bressan-SICON}, they may not be robust with respect to arbitrary small measurement errors on the state, which is crucial for the application. To have a robustness with respect to arbitrary small measurement errors on the state, as shown in \cite{1967-Hermes-AP} (see also \cite{1994-Coron-Rosier-JMSEC}), the good definition of the solutions for the closed-loop system \eqref{eq:closed-loop} are the Filippov solutions, i.e. the solution of
\begin{gather}
\label{def:Filippov}
\dot x \in Y(x) := \underset{\varepsilon>0}{\mathbin{\scaleobj{1.75}{\cap}}}
\underset{N\in \mathcal{N}}{\mathbin{\scaleobj{1.75}{\cap}}}\overline{\text{conv}}\left[X\left(\left((x+\varepsilon B)\cap \mathcal{D}'\right)\setminus N\right)\right],
\end{gather}
where
\begin{itemize}
\item $B$ is the unit ball of $\RR^4$;
\item for a set $A$, $\overline{\text{conv}}[A]$ is the smaller closed convex set containing $A$;
\item $\mathcal{N}$ is the set of subsets of $\RR^4$ of zero Lebesgue measure.
\item $X(x):= H(x, u(x)).$
\end{itemize}
Let us recall that $x:I\subset \RR\rightarrow \RR^4$, $t\in I \mapsto x(t)\in \RR^4$ (where $I$ is an interval of $\RR$) is a solution of
\eqref{def:Filippov} if $x\in W^{1,\infty}_{\text{loc}}(I)$ and is such that
\begin{gather}
\dot x(t) \in Y(x(t)) \text{ for almost every } t\in I.
\end{gather}
 For references about Filippov solutions, let us mention, in particular, \cite{1960-Filippov-MS, 1988-Filippov-book}
and \cite[Chapter 1]{2005-Bacciotti-Rosier-book}. For the definition of stability, global attractor and asymptotic stability, we use again Definition~\ref{def-stab} (with $\mathcal{D}'$ instead of $\mathcal{D}$) and take now  into account all the solutions in the Filippov sense in this definition. The motivation for using Filippov solutions is given in  \cite[Proposition 1.4]{1994-Coron-Rosier-JMSEC}. The global asymptotic stability in this Filippov sense implies the existence of a Lyapunov function \cite{1998-Clarke-Ledyaev-Stern-1998}; see also \cite[Lemma 2.2]{1994-Coron-Rosier-JMSEC}. This   automatically gives some robustness properties with respect to (small) perturbations (including small measurement errors on the state), which is precisely the goal of feedback laws. In fact, for many feedback laws constructed in this article, an explicit Lyapunov function will be given, which allows to quantify this robustness.

Let us emphasize that in our case the Filippov solutions of our closed-loop system have the  following properties
\begin{gather}
\label{0donne0}
\left((E(0),F(0)) =(0,0) \right)\Longrightarrow\left((E(t),F(t))=(0,0)\; \forall t\geq 0\right),
\\
\label{tout-de-suite>0}
\left((E(0),F(0))\not =(0,0) \right)\Longrightarrow\left(E(t)>0, \, M(t)>0,\, F(t)>0 \; \forall t> 0\right).
\end{gather}
 From now on, the solutions of the closed-loop systems considered in this article are always the Filippov solutions.
	
	\begin{proposition}[See\cite{almeida2022optimal}:
		Stability properties of the system \eqref{eq:S1E1}-\eqref{eq:S1E4}]
\label{prop-constant-feedback}
Let us assume that
\begin{gather}
\label{R0>1}
\mathcal{R}_0>1.
\end{gather}
Then the following properties hold.
\begin{enumerate}
\item If $u=0,$ we have two equilibria:
			\begin{itemize}
				\item the extinction equilibrium $\textbf{0}$, where $E=F=M=M_s=0$ ,which is linearly unstable;
				\item the persistence equilibrium
				\begin{align}
					&{E^*} = K(1 -\frac{1}{\mathcal{R}_0}),\label{eq:equilibreE-2}  \\
&{M^*} = \frac{(1-\nu)\nu_E}{\delta_M}{E^*}\label{eq:equilibreM-2},\\
&{F^*} = \frac{\nu\nu_E}{\delta_F}E^*,\\
&M_s^*=0, \label{eq:equilibreMs-2}
				\end{align}
				which is locally asymptotically stable.
			\end{itemize}
			\item If $u\geq 0$, then the corresponding solution $(E,M,F,M_s)$ to System \eqref{eq:S1E1}-\eqref{eq:S1E4} enjoys the following stability property:
			\begin{align}
				\left\{
				\begin{aligned}
					&E(0)\in(0,E^*], \\
					& M(0)\in(0,M^*], \\
					&F(0)\in(0,F^*], \\
					& M_s(0)\geq 0,
				\end{aligned}
				\right.\implies \left\{
				\begin{aligned}
					&E(t)\in(0,E^*],  \\
					& M(t)\in(0,M^*], \\
					&F(t)\in(0,F^*], \\
					& M_s(t)\geq 0,
				\end{aligned}
				\right. \;\;\mbox{for all}\; \;t\geq 0.\label{eq:monotony}
			\end{align}
			Let
			\begin{align}
				U^* = \mathcal{R}_0\frac{K(1-\nu)\nu_E\delta_s}{4\gamma_s\delta_M}(1-\frac{1}{\mathcal{R}_0})^2.
			\end{align}
			If $u(\cdot)$ denotes a constant control function equal to some $\overline{U}>U^*$ for all $t\geq 0$, then the corresponding solution  $(E(t), M(t), F(t), M_s(t))$ converges to $(0,0,0,\overline{U}/\delta_s)$ as $t\to\infty$.
		\end{enumerate}
	\end{proposition}

Concerning the global asymptotic stability of $\textbf{0}$ for the system \eqref{eq:S1E1}-\eqref{eq:S1E4} in $\mathcal{D}':=[0,+\infty)^4$, using a Lyapunov approach, one can get  the following theorem.
\begin{theorem}\label{th-u=0-with-Ms}
	Let $u=0$. If  $\mathcal{R}_0<1$,  then   $\textbf{0}$  is globally asymptotically  stable in $\mathcal{D}'$ for the system \eqref{eq:S1E1}-\eqref{eq:S1E4}.
\end{theorem}
\begin{proof}
Let $x= (E,M,F, M_s)^T$. We are going to conclude by applying Lyapunov's second theorem.
To do so, a candidate  Lyapunov function  is $V: \mathcal{D}'\rightarrow\RR_+$, $x\mapsto V(x)$, defined by
\begin{equation}
\label{defV-u=0}
	V(x) :=\frac{1+\mathcal{R}_0}{1-\mathcal{R}_0} E +\frac{2 \beta_E}{\delta_F(1-\mathcal{R}_0)} F + M + M_s.
\end{equation}
Note that, since $\mathcal{R}_0<1$,
\begin{gather}
\label{V>0-sec2}
V(x)>V(\textbf{0})=0, \; \forall x\in \mathcal{D}'\setminus\{\textbf{0}\},
\\
\label{Vinfty-sec2}
V(x)\rightarrow +\infty \text{ as } |x|\rightarrow +\infty \text{ with } x\in \mathcal{D}'.
\end{gather}
Moreover, along the trajectories of \eqref{eq:S1E1}-\eqref{eq:S1E4},
\begin{multline}
\label{dotV-sec2}
	 \dot V (x)=-(\nu\nu_E+\delta_E)E -\frac{\beta_E}{K}\frac{1+\mathcal{R}_0}{1-\mathcal{R}_0}FE
-\delta_MM -\beta_E  F-\delta_sM_s\\
-\frac{2\beta_E\nu \nu_E}{\delta_F(1-\mathcal{R}_0)}\frac{\gamma_s M_s}{M+\gamma_s M_s} E\text{, if } M+M_s\not =0.
\end{multline}
 From \eqref{defV-u=0} and \eqref{dotV-sec2}, one gets
\begin{equation}
\label{dotV<-sec2c}
	 \dot V (x)\leq -c_0 V(x) \text{ if } M+M_s\not =0,
\end{equation}
with
\begin{equation}
\label{decay}
c_0 := \min \left\{ \frac{(\nu \nu_E +\delta_E) (1-\mathcal{R}_0)}{1+\mathcal{R}_0} , \frac{\delta_F (1-\mathcal{R}_0)}{2} ,  \delta_M , \delta_s
\right\}
\end{equation}

\noindent
Let us point out that, for every solution $t\mapsto x(t)=(E(t),M(t),F(t),M_s(t))^T$ of the closed-loop system \eqref{eq:S1E1}-\eqref{eq:S1E4} defined at time $0$ and such that $x(0)\in \mathcal{D}'$,
\begin{gather}
\label{Ms(0)>0Ms>0}
\left(M(0)+M_s(0)>0\right)\Longrightarrow \left(M(t)+M_s(t)>0, \; \forall t>0\right),
\\
\label{x(0)=0givesx(t)=0}
\left(x(0)=0\right)\Longrightarrow\left(x(t)=0, \; \forall t\geq 0\right).
\end{gather}
 From \eqref{0donne0}, \eqref{tout-de-suite>0}, \eqref{V>0-sec2}, \eqref{dotV<-sec2c}, \eqref{Ms(0)>0Ms>0} and \eqref{x(0)=0givesx(t)=0}, one has, for every solution $t\mapsto x(t)=(E(t),M(t),F(t),M_s(t))^T$ of the closed-loop system \eqref{eq:S1E1}-\eqref{eq:S1E4} defined at time $0$ and such that $x(0)\in \mathcal{D}'$,
 \begin{gather}
 V(x(t))\leq V(x(0)) e^{-c_0 t}\; \forall t\geq 0,
\end{gather}
which, together with  \eqref{V>0-sec2} and \eqref{Vinfty-sec2}, concludes the proof of Theorem~\ref{th-u=0-with-Ms} (and even shows the global exponential stability and provides an estimate on the exponential decay rate $c_0$ given by \eqref{decay}).
\end{proof}
\begin{remark}
\label{rem-strict-Lyapunov}
Note that Theorem~\ref{th-u=0-with-Ms} implies Theorem~\ref{th-as-without-Ms} in the case  $\mathcal{R}_0<1$ and our proof of Theorem~\ref{th-u=0-with-Ms} provides, for this case,  a (strict) Lyapunov function which is just
\begin{gather}
\tilde V((E,M,F)^T):=\frac{1+\mathcal{R}_0}{1-\mathcal{R}_0} E +\frac{2 \beta_E}{\delta_F(1-\mathcal{R}_0)} F + M.
\end{gather}
It would be interesting to provide Lyapunov functions for the two remaining cases  $\mathcal{R}_0=1$ and $\mathcal{R}_0>1$.

\end{remark}
\section{Global stabilization by feedback of the extinction equilibrium}
\subsection{Backstepping feedback}
\label{backsteppingSubsection}
 For the backstepping method, the control system  has the following structure:
\begin{align}
	&\dot{x}_1 = f(x_1,x_2),\label{eq:e1}\\&\dot{x}_2 = u - g(x_1,x_2), \label{eq:e2}
\end{align}
where the state is $x = (x_1,x_2)\in \RR^p\times\RR^m$ and the control is $u\in\RR^m$. The key and classical theorem for backstepping is the following one (see, for instance,  \cite[Theorem 19.2, page 110]{1992-Bacciotti-book} or \cite[Theorem 12.24, page 334]{coron2007control}).
\begin{theorem}\label{th-backstepping}
	Assume that $f$ and $g$ are of class $\mathcal{C}^{1}$ and that for the control system
	\begin{align}
		\dot{x}_1 = f(x_1,v),
	\end{align}
	where the state is $x_1\in\RR^p$ and the control is $v\in\RR^m$, ${\bf 0} \in\RR^p$ can be globally asymptotically stabilized by means of a  feedback law $x_1\in \RR^p \mapsto  v(x_1)\in \RR^m$ of class $\mathcal{C}^1$. Then,  for  the control system \eqref{eq:e1}-\eqref{eq:e2}, ${\bf 0} \in\RR^p\times\RR^m$ can be globally asymptotically stabilized by means of a continuous feedback law $x\in \RR^p\times\RR^m\mapsto  u(x)\in \RR^m$ .
\end{theorem}

Let $x:= (E, M, F)^T$. One way to rewrite the dynamics \eqref{eq:S1E1}-\eqref{eq:S1E4} is
\begin{equation}
	\left\{
	\begin{aligned}
		&\dot{x}= f(x,M_s),\\
		&\dot{M_s}  = u-\delta_sM_s,
	\end{aligned}
	\right.
\end{equation}
where
\begin{align}
	f(x,M_s) := \left( \begin{array}{ccc} \beta_E F \left(1-\frac{E}{K}\right) - \big( \nu_E + \delta_E \big) E\\  (1-\nu)\nu_E E - \delta_M M\\  \nu\nu_E E \frac{M}{M+\gamma_s M_s} - \delta_F F
	\end{array}
	\right).
\end{align}
As  $f$  is not of class $\mathcal{C}^1(\mathcal{D}\times [0,+\infty))$ and the feedback law has to  be non-negative, we cannot directly apply the backstepping theorem. However, to build the feedback law  we use the classical Lyapunov approach of the proof of Theorem~\ref{th-backstepping} (see, for example, \cite[pages 334--335]{coron2007control})  allowing us to select an appropriate control. Unfortunately, the control that we get with this approach is not positive all the time. To get around this, using the same Lyapunov function, we propose a new feedback law that is non-negative, decreases the Lyapunov function and leads to global asymptotic stability of the extinction equilibrium.

 First, consider the control system $\dot{x}= f(x,M_s)$ with  the state being $x\in\mathcal{D}$ and the control being  $M_s\in [0,+\infty)$. We assume that  $M_s$ is of the form $M_s = \theta M$ and  study the closed-loop system
\begin{align}
	\dot{x}= f(x,\theta M).
	\label{eq:backeq}
\end{align}
We have
\begin{equation}
	\left\{
	\begin{aligned}
		&\dot{E}  = \beta_E F \left(1-\frac{E}{K}\right) - \big( \nu_E + \delta_E \big) E,  \\
		&\dot{M} = (1-\nu)\nu_E E - \delta_M M,  \\
		&\dot{F} =
		\frac{\nu\nu_E}{1+\gamma_s \theta}E - \delta_F F.
	\end{aligned}
	\right.
	\label{eq:feedQ}
\end{equation}
It is a  smooth dynamical system on $\mathcal{D}=[0,+\infty)^3$ which is also a positively invariant set for this dynamical system. \\ Setting the right hand side of \eqref{eq:feedQ} to zero we obtain the equilibrium $\textbf{0}\in [0,+\infty)^3$ and the non-trivial equilibrium $x^{**}=(E^{**}, M^{**}, F^{**})$ given by
\begin{align}
&{E}^{**} = K(1 -\frac{1}{\mathcal{R}(\theta)}),\label{eq:equilibreE-3}
\\
& {M}^{**} = \frac{(1-\nu)\nu_E}{\delta_M}{E}^{**} \label{eq:equilibreM-3},
\\
&{F}^{**} = \frac{\nu\nu_E}{\delta_F(1+\gamma_s\theta)}{E}^{**},
\end{align}
where the offspring number is now

\begin{align}
	\mathcal{R}(\theta) := \frac{\beta_E\nu\nu_E}{\delta_F(1+\gamma_s\theta)(\nu_E+\delta_E)} =  \frac{\mathcal{R}_0}{1+\gamma_s\theta}\label{eq:basenumberteta}.
\end{align}
Note that if $\mathcal{R}(\theta)\leq 1$, $\textbf{0} \in \RR^3$ is the only equilibrium point of the system in $\mathcal{D}$.

Our next proposition shows that the feedback law $M_s=\theta M$ stabilizes our control system $\dot{x}= f((x^T,M_s)^T)$ if $\mathcal{R}(\theta)<1$.
\begin{proposition}
\label{prop-case-thetaM-GAS}
	Assume that
\begin{align}
	\mathcal{R}(\theta)<1 \label{eq:inequalitytheta}.
\end{align} Then $\textbf{0}$  is globally asymptotically stable in $\mathcal{D}$ for  system \eqref{eq:backeq}.
\end{proposition}
\begin{proof}
We apply Lyapunov's second theorem. To do so, we define

\begin{equation}
        \begin{aligned}
        &V:  x\in [0,+\infty)^3  \mapsto V(x)\in \RR_+, \\
	&V(x) := \frac{1+\mathcal{R}(\theta)}{1-\mathcal{R}(\theta)} E + M + \frac{2\beta_E}{\delta_F(1-\mathcal{R}(\theta))}F.
\end{aligned}
        \end{equation}

\noindent As \eqref{eq:inequalitytheta} holds,
\begin{gather}
\text{ V is of class $\mathcal{C}^1$},
\\
\label{V>0}
V(x)>V((0,0,0)^T)=0, \;\forall x \in [0, +\infty)^3\setminus\{(0,0,0)^T\},
\\
\label{Vinftyatinfty}
\text{$V(x)\to +\infty$  when $\norm{x}\to +\infty$ with $x\in\mathcal{D}$.}
\end{gather}
We have
\begin{equation}
	\dot{V}(x) = \nabla V(x)\cdot f(x, \theta M)= \begin{pmatrix}\frac{1+\mathcal{R}(\theta)}{1-\mathcal{R}(\theta)}\\  1 \\ \frac{2\beta_E}{\delta_F(1-\mathcal{R}(\theta))}
\end{pmatrix}^T
\cdot
\begin{pmatrix} \beta_E F \left(1-\frac{E}{K}\right) - a E\\  c E - \delta_M M\\   \frac{\nu\nu_E}{1+\gamma_s\theta} E - \delta_F F
\end{pmatrix}.
\end{equation}
 So
\begin{equation}
\label{dotV-value}
	\dot{V}(x) = -\beta_E F-\delta_M M -\frac{1+\mathcal{R}(\theta)}{1-\mathcal{R}(\theta)}\frac{\beta_E}{K} FE - (\nu\nu_E+\delta_E)E.
\end{equation}
Then, using once more \eqref{eq:inequalitytheta}, we get the existence of $c>0$ such that
\begin{align}
\label{dotV<0}
    \dot{V}(x) \leq -c V(x), \; \forall x \in [0, +\infty)^3.
\end{align}
This concludes the proof of Proposition~\ref{prop-case-thetaM-GAS}.
\end{proof}
Let us define
\begin{align}
	\psi := \frac{2\beta_E\nu\nu_E}{\delta_F(1-\mathcal{R}(\theta))(1+\gamma_s\theta)},
\end{align}
 and, for $\alpha${for $\alpha$ and $\beta_s$ (the latter having dimension of a rate )} chosen in $(0,+\infty)$, the map $G: \mathcal{D}':=[0,+\infty)^4\rightarrow \RR$, $(x^T,M_s)^T\mapsto G((x^T,M_s)^T)$ by
\begin{multline}
		G((x^T,M_s)^T):=	\frac{\gamma_s\psi E(\theta M + M_s)^2}{\alpha(M+\gamma_sM_s)(3\theta M + M_s)}\\
+ \frac{((1-\nu)\nu_E\theta E -\theta \delta_M M)(\theta M +3M_s)}{3\theta M + M_s}
\\ +\delta_sM_s + {\frac{\beta_s}{\alpha}}(\theta M-M_s) \text{, if } M+M_s\not=0,
\end{multline}
\begin{equation}
G((x^T,M_s)^T):=0 \text{, if } M+M_s=0.
\end{equation}
 Finally, let us define the feedback law $u: \mathcal{D}'\rightarrow [0,+\infty)$,
 $(x^T,M_s)^T\mapsto u((x^T,M_s)^T)$, by
\begin{align}
		u((x^T,M_s)^T):=\max\left(0,G((x^T,M_s)^T)\right). \label{eq:backcontr}
\end{align}	
Note that $u$, which is Lebesgue measurable,  is not continuous in $\mathcal{D}'$. However
\begin{gather}
\label{utens0-at-0}
\text{there exists $C>0$ such that } |u(y)|\leq C \|y\| \; \forall y\in \mathcal{D}'.
\end{gather}
Property \eqref{utens0-at-0} is important for the applications since it implies that the density $u$ of sterile males released is going to be small when the state is close to $\textbf{0}$. For instance, this is essential to reduce the number of mosquitoes necessary for a long term intervention and also to allow using the sterile mosquitoes which are no longer needed in an area where the population is already close to zero, to intervene in other zones.

This is in contrast with the constant control in Proposition~\ref{prop-constant-feedback}. Property \eqref{utens0-at-0} also implies that $
u\in L^\infty_{\text{loc}}(\mathcal{D}')$, which allows to consider Filippov solutions for the closed-loop system, i.e. the system \eqref{eq:S1E1}-\eqref{eq:S1E4}  with the feedback law \eqref{eq:backcontr}.

The next theorem shows that the feedback law \eqref{eq:backcontr} stabilizes the control system \eqref{eq:S1E1}-\eqref{eq:S1E4}.
\begin{theorem}
\label{thm-backstepping}
Assume that (\ref{eq:inequalitytheta}) holds. Then $\textbf{0}\in \mathcal{D}'$ is globally asymptotically  stable in $\mathcal{D}'$ for system \eqref{eq:S1E1}-\eqref{eq:S1E4}  with the feedback law \eqref{eq:backcontr}.
\end{theorem}
\begin{proof}
Let  us define $W:\mathcal{D}'\rightarrow \RR$ by
\begin{gather}
\label{eq:lyapunov-functionW}
	W((x^T,M_s)^T) :=  V(x) + {\alpha}\frac{(\theta M-M_s)^2}{\theta M + M_s} \text{, if } M+M_s\not = 0,
\\
W((x^T,M_s)^T) := V(x)  \text{, if } M+M_s = 0.
\end{gather}
We have
\begin{gather}
W \text{ is continuous},
\\
W \text{ is of class $\mathcal{C}^1$ on $\mathcal{D}'\setminus \left\{(E,M,F,M_s)^T\in \mathcal{D}';\; M+M_s=0\right\}$},
\\
\label{Winftyatinfty}
W((x^T,M_s)^T)\to +\infty \mbox{, as } \norm{x} +M_s\to +\infty,\nonumber\\\text{ with $x\in \mathcal{D}$ and $M_s\in [0,+\infty)$},
\\
\label{W>0}
W((x^T,M_s)^T)>W(\textbf{0})=0, \; \forall (x^T,M_s)^T\in \mathcal{D}'\setminus\{\textbf{0}\}.
\end{gather}
 From now on, and until the end of this proof, we assume that $(x^T,M_s)^T$ is in $\mathcal{D}'$ and until \eqref{dotW<-case-2-2} below we further assume that
\begin{gather}
\label{MMsnotboth0}
(M,M_s)\not =(0,0).
\end{gather}
One has
\begin{equation*}
\begin{array}{rcl}
	\dot{W}((x^T,M_s)^T) &= & \nabla V(x)\cdot f(x, M_s)+{\alpha}(\theta M-M_s)\\
&&\displaystyle \frac{2(\theta \dot{M}-\dot{M}_s)(\theta M + M_s)-(\theta\dot{M}+\dot{M_s})(\theta M-M_s)}{(\theta M + M_s)^2} \\ &= &\nabla V(x)\cdot f(x, \theta M) + \nabla V(x)\cdot(f(x,M_s)-f(x,\theta M))
\\  && \displaystyle + {\alpha}(\theta M-M_s)\frac{\theta\dot{M}(\theta M + 3M_s)-\dot{M}_s(3\theta M+M_s)}{(\theta M + M_s)^2}.
\end{array}
\end{equation*}

\begin{multline}
\nabla V(x)\cdot (f(x,M_s)-f(x,\theta M))=
\\
\begin{pmatrix}\displaystyle \frac{1+\mathcal{R}(\theta)}{1-\mathcal{R}(\theta)} \\ 1 \\ \displaystyle \frac{2\beta_E}{\delta_F(1-\mathcal{R}(\theta))}
\end{pmatrix}^T
\cdot
\begin{pmatrix}
0\\ 0\\  \displaystyle \frac{\nu\nu_E\gamma_sE( \theta M-M_s)}{(M +\gamma_sM_s)(1+\gamma_s\theta)}
\end{pmatrix}
=
\\
\displaystyle \frac{\psi\gamma_sE(\theta M-M_s)}{M+\gamma_sM_s},
\end{multline}

\begin{eqnarray}	
\dot{W}((x^T,M_s)^T)&= & \nabla V(x)\cdot f(x, \theta M) + \alpha\frac{(\theta M-M_s)}{(\theta M + M_s)^2}
\nonumber
\\&& \Big[\frac{(\nabla V(x)\cdot(f((x^T,M_s)^T)-f(x,\theta M)))(\theta M + M_s)^2}{\alpha(\theta M-M_s)}\nonumber
\\
&& \phantom{bbbbbbb} +\theta \dot{M}(\theta M +3M_s)-\dot{M}_s(3\theta M + M_s)\Big]\nonumber\\ &=&
\dot{V}(x) + \alpha\frac{(\theta M-M_s)}{(\theta M + M_s)^2}\Big[ \frac{\psi\gamma_sE(\theta M + M_s)^2}{\alpha(M+\gamma_sM_s)} \\&& + ((1-\nu)\nu_E\theta E -\theta \delta_M M)(\theta M +3M_s)\nonumber
\\&&-u(3\theta M + M_s)+\delta_s{M}_s(3\theta M + M_s)\Big].
\label{dotW=}
\end{eqnarray}

We take $u$ as given by \eqref{eq:backcontr}.\\ Therefore, in case
\begin{multline}
\label{psigamma>}
	\frac{\psi\gamma_sE(\theta M + M_s)^2}{\alpha(M+\gamma_sM_s)} + ((1-\nu)\nu_E\theta E -\theta \delta_M M)(\theta M +3M_s)
\\+
\delta_sM_s(3\theta M + M_s) + {\frac{\beta_s}{\alpha}}(\theta M-M_s)(3\theta M+M_s)>0,
\end{multline}
we have
	\begin{multline*}
		u=\frac{1}{3\theta M + M_s}\Big[	\frac{\psi\gamma_sE(\theta M + M_s)^2}{\alpha(M+\gamma_sM_s)} + ((1-\nu)\nu_E\theta E -\theta \delta_M M)(\theta M +3M_s)
\\+\delta_sM_s(3\theta M + M_s) + {\frac{\beta_s}{\alpha}}(\theta M-M_s)(3\theta M+M_s)\Big],
	\end{multline*}
which, together with \eqref{dotW=}, leads to
\begin{align}
\label{dotW<-case-1}
	\dot{W}((x^T,M_s)^T)= \dot{V}(x) -\beta_s\frac{(\theta M-M_s)^2(3\theta M+M_s)}{(\theta M + M_s)^2}.
\end{align}
Otherwise, i.e. if \eqref{psigamma>} does not hold,
\begin{multline}
	\frac{\psi\gamma_sE(\theta M + M_s)^2}{\alpha(M+\gamma_sM_s)}  + ((1-\nu)\nu_E\theta E -\theta \delta_M M)(\theta M +3M_s) \\
+\delta_sM_s(3\theta M +M_s)
+ {\frac{\beta_s}{\alpha}}(\theta M-M_s)(3\theta M+M_s)\leq 0,\label{eq:inequalityueu}
\end{multline}
so, by \eqref{eq:backcontr},
\begin{gather}
\label{u=0}
u=0.
\end{gather}
We consider two cases:\\
\textbf{\underline{Case 1:}}  $\theta M > M_s$\\
\noindent Using \eqref{dotW=}, \eqref{eq:inequalityueu} and \eqref{u=0}
\begin{align}
\label{dotW<-case-2-1}
	\dot{W}((x^T,M_s)^T)&\leq  \dot{V}(x) -\beta_s\frac{(\theta M-M_s)^2(3\theta M+M_s)}{(\theta M + M_s)^2}.
\end{align}

\noindent
\textbf{\underline{Case 2:}}  $\theta M \leq  M_s$\\
Using once more \eqref{dotW=} and \eqref{u=0}
\begin{multline}
\label{case-2-dotW}
\dot{W}((x^T,M_s)^T) =\dot{V}(x) + \alpha\frac{(\theta M-M_s)}{(\theta M + M_s)^2}\Big[ \frac{\psi\gamma_sE(\theta M + M_s)^2}{\alpha(M+\gamma_sM_s)}  \\
+ \theta ((1-\nu)\nu_EE-\delta_M M)(\theta M +3M_s)+\delta_s{M}_s(3\theta M + M_s)\Big].
\end{multline}
Using \eqref{deltas>deltaM}
\begin{align*}
	-\delta_M M(\theta M +3M_s)+\delta_s{M}_s(3\theta M + M_s)\geq \delta_M(M_s-\theta M)(M_s+\theta M),
\end{align*}
which, together with \eqref{case-2-dotW}, implies that
\begin{equation}
\label{dotW<-case-2-2}
		\dot{W}((x^T,M_s)^T) \leq \dot{V}(x) - \alpha\delta_M\frac{(\theta M-M_s)^2}{(\theta M + M_s)}.
\end{equation}

To summarize, using \eqref{dotV<0}, \eqref{dotW<-case-1}, \eqref{dotW<-case-2-1} and \eqref{dotW<-case-2-2}, one gets the existence of $c'>0$, independent of $(x^T,M_s)^T\in \mathcal{D}'$, such that
\begin{gather}
\label{dotW<-final}
\dot{W}((x^T,M_s)^T) \leq -c'{W}((x^T,M_s)^T) \text{ if }M+M_s\not=0.
\end{gather}
Since one still has \eqref{0donne0}, \eqref{tout-de-suite>0}, \eqref{Ms(0)>0Ms>0} and \eqref{x(0)=0givesx(t)=0} (for $x=(x^T,M_s^T)^T$), this proves  Theorem~\ref{thm-backstepping} as in the proof of Theorem \ref{th-u=0-with-Ms} (and, again, even gives the global exponential stability and provides an estimate on the exponential decay rate).
\end{proof}
\begin{remark}
	It is important to note that the backstepping feedback control \eqref{eq:backcontr} does not depend on the environmental capacity $K$, which is can also be an interesting feature for the field applications.
\end{remark}
\subsubsection{Numerical simulations}
\label{subsec-sim-back}
The numerical simulations of the  dynamics when applying the feedback \eqref{eq:backcontr} are shown in  figure~\ref{fig:simulation1}.
The parameters  we use are set in  table \ref{eq:tableparametre}. The condition \eqref{eq:inequalitytheta} gives  $\theta >75.5625$. We fix   $K=22200 \text{ ha}^{-1}$ and we consider the persistence equilibrium as initial condition. That gives $E^0 = 21910, M^0 = 5587, F^0 = 13419
$ and $M_s^0 = 0$.  We take $\theta =220$, $\alpha=13$ and   $\beta_s=1\text{ Day}^{-1}$.

\begin{table}[H]
    \centering
	\setlength{\tabcolsep}{0.03cm}
	\begin{tabular}{|c|c|c|c|c|}
		\hline
		Parameter  &  Name &Value interval& Chosen value & Unity \\
		\hline
		$\beta_E$ & Effective fecundity &7.46-14.85& 10& Day$^{-1}$\\
		\hline
		$\gamma_s$ & Mating competitiveness of
		sterilized males  &0-1&1&- \\
		\hline
		$\nu_E$ & Hatching parameter &0.005-0.25& 0.05&Day$^{-1}$\\
		\hline
		$\delta_E$& Mosquitoes in aquatic
		phase death rate & 0.023-0.046&0.03&Day$^{-1}$\\
		\hline
		$\delta_F$& Female death rate &0.033-0.046& 0.04&Day$^{-1}$\\
		\hline
		$\delta_M$ & Males death rate &0.077-0.139&  0.1&Day$^{-1}$\\
		\hline
		$\delta_s$ & Sterilized male death rate& & 0.12&Day$^{-1}$\\
		\hline
		$\nu$ & Probability of emergence& & 0.49&\\
		\hline
	\end{tabular}
		\caption{Value intervals for the parameters of system \eqref{eq:S1E1}-\eqref{eq:S1E4} (see \cite{strugarek2019use})}
	\label{eq:tableparametre}
\end{table}

\begin{figure}[H]
	\centering
	\begin{subfigure}[H]{0.45\textwidth}
		\centering
		\includegraphics[width=\textwidth]{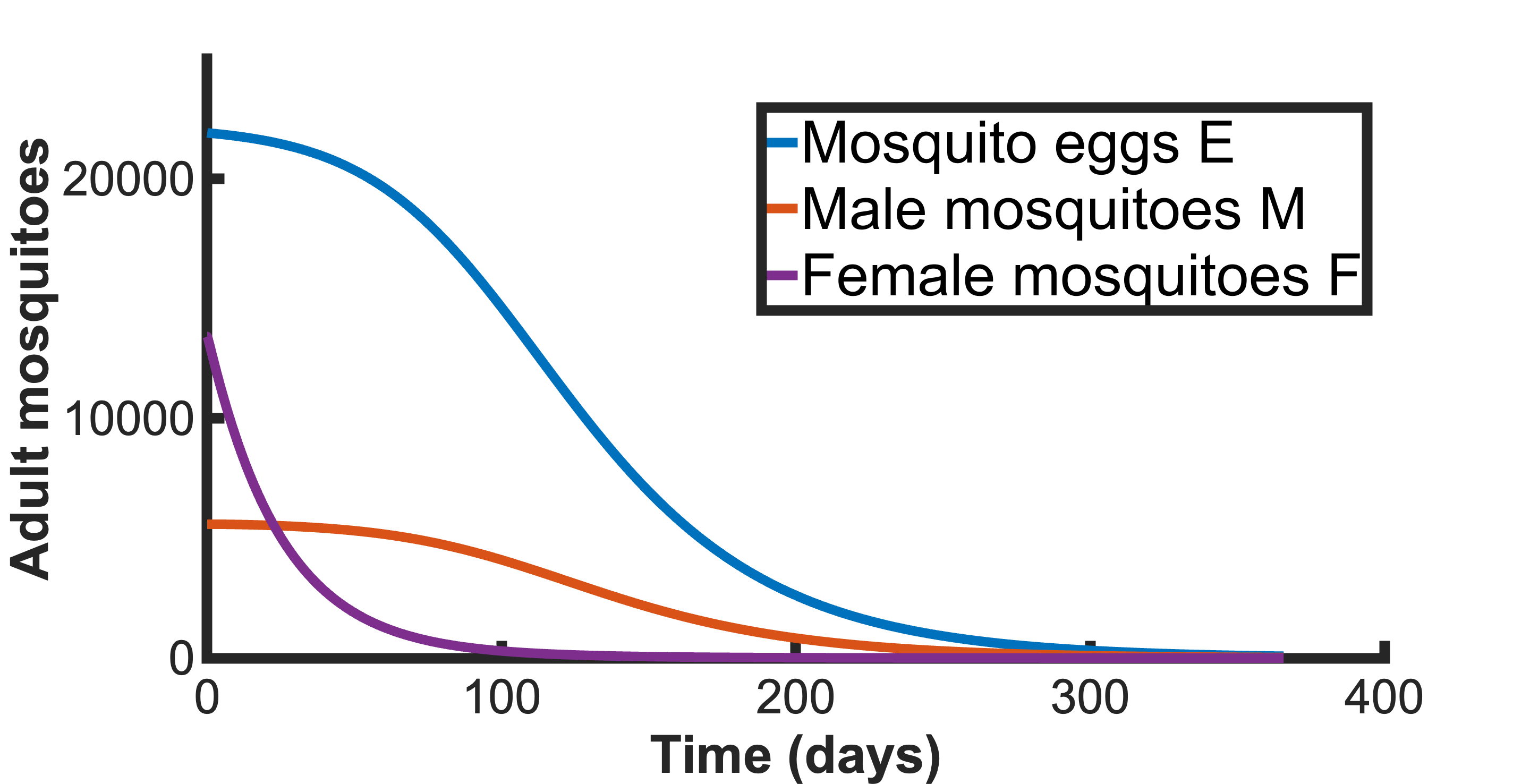}
		\caption{Plot of $E, M$ and $F$}
		\label{fig:evolutionEMF}
	\end{subfigure}
	\hfill
	\begin{subfigure}[H]{0.45\textwidth}
		\centering
		\includegraphics[width=\textwidth]{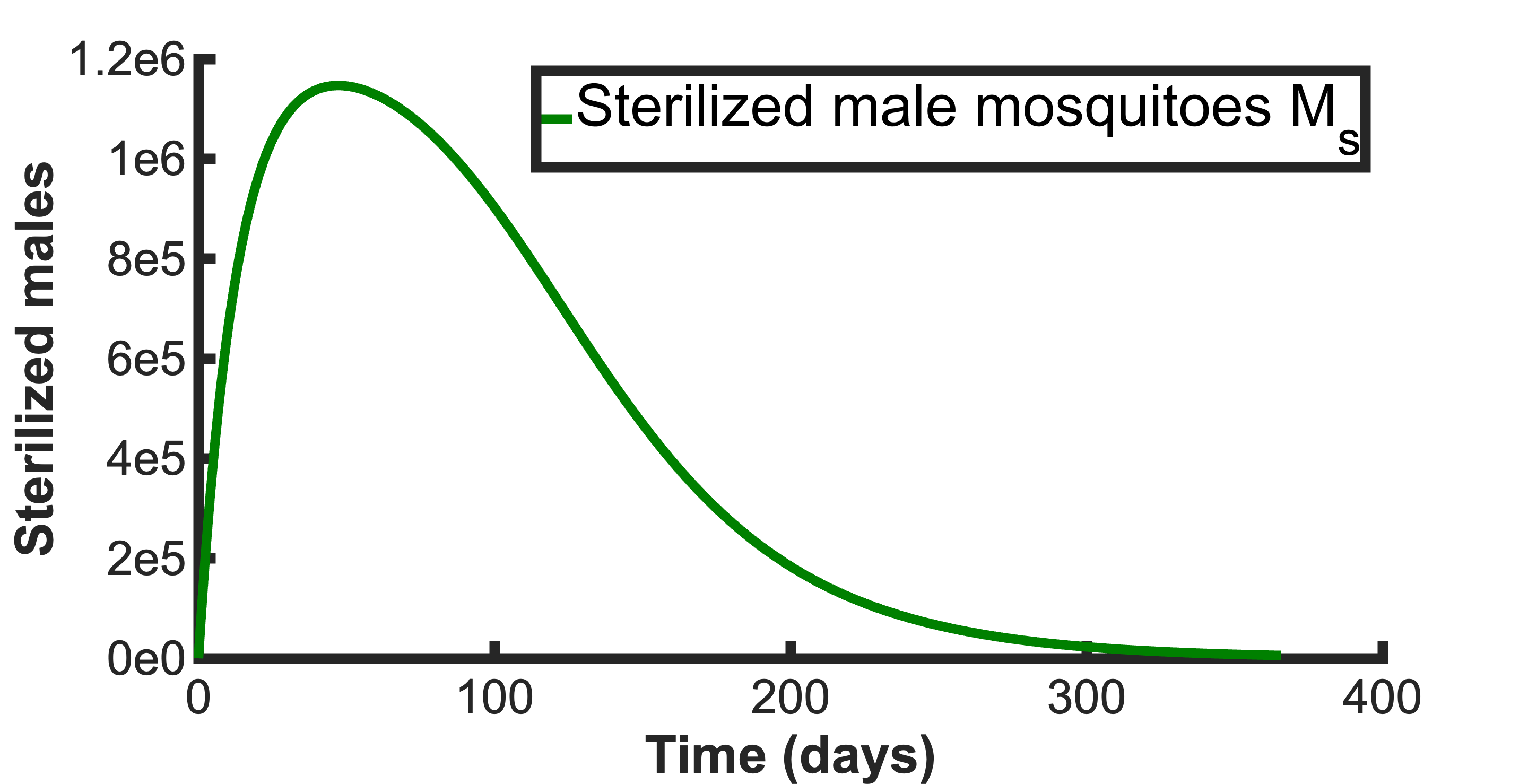}
		\caption{Plot of $M_s$}
		\label{fig:EvolutionMs-1}
	\end{subfigure}
	\hfill
	\begin{subfigure}[H]{0.5\textwidth}
		\centering
		\includegraphics[width=\textwidth]{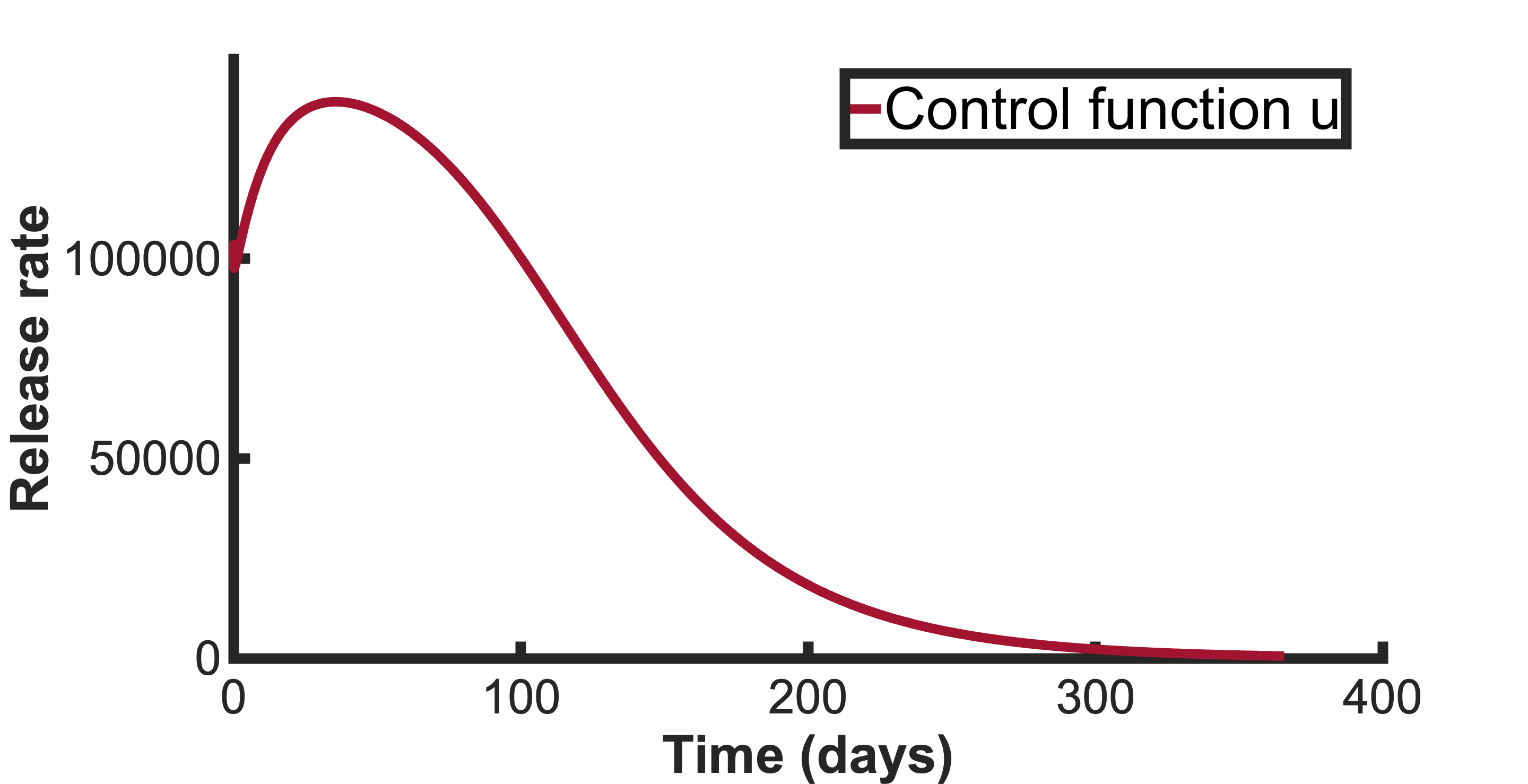}
		\caption{Plot of the control $u$}
		\label{fig:five over x-1}
	\end{subfigure}
	\begin{fig}
		(a): Plot of $E,M$ and $F$ when applying the feedback \eqref{eq:backcontr}, the initial condition being the the persistence equilibrium.  (b): Plot of $M_s$. (c): Plot of the feedback control function $u$.
\label{fig:simulation1}
	\end{fig}
	
\end{figure}
In this case, with $t_f = 360$ days,
\begin{align}
	\int_{0}^{t_f} u(t)\; dt \approx 18\;\mbox{millions}.
\end{align}

\subsubsection{Robustness test}\label{see:robusttessforback}

		 To analyze the robustness of our feedback law we use the following protocol:
		 the feedback law is given by \eqref{eq:backcontr} with fixed values of the parameters
		 corresponding to the ones chosen in table \ref{eq:tableparametre},  but for computing the
		real dynamics of the system \eqref{eq:closed-loop}  we consider
simultaneous random perturbations of the system parameters with the following distribution
\begin{align}
\begin{split}
        \hat\beta_E  &\sim\mathcal{U}(7.46, 14.85),\\
        \hat\nu_E &\sim  \mathcal{U}(0.005, 0.25),\\
        \hat\delta_E &\sim\mathcal{U}(0.023, 0.046),\\
        \hat\delta_F &\sim \mathcal{U}(0.033, 0.046),\\
        \hat\delta_M &\sim \mathcal{U}(0.077, 0.139),\\
         \hat\delta_s &\sim \mathcal{U}(0.077, 0.139),\\
        \hat\gamma_s &\sim\mathcal{U}(0.5, 1.0),
\label{eq:sim_params_random}
\end{split}
\end{align}
where $\mathcal{U}(a,b)$ is the uniform distribution on interval $[a,b]$.

Figure \ref{eq:Robtest1} shows  $200$ simulations with random initial conditions in $[0, 10K]^4$.
			
\begin{figure}[H]
		\centering
		\includegraphics[width=0.7\textwidth]{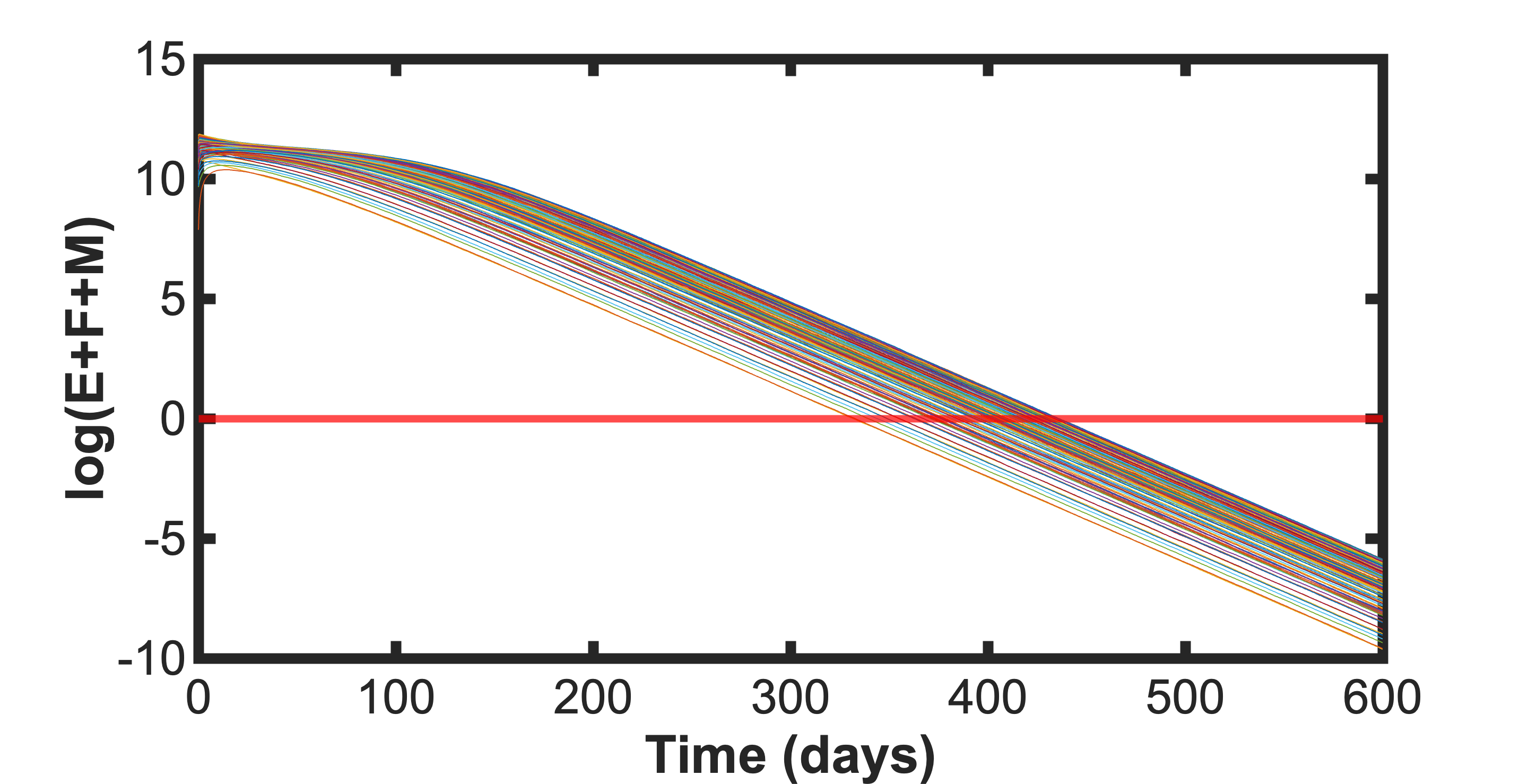}
		\begin{fig}
		  Evolution of $\log(E+F+M)$  for $200$ random initial conditions when applying feedback \eqref{eq:backcontr} computed with $\theta = 220$, $\alpha = 13$,  $\beta=1\text{ Day}^{-1}$ and the parameters given in table~\ref{eq:tableparametre}, which are not the parameters. used for simulating the mosquito dynamics (the latter being taken randomly for each simulation according to \eqref{eq:sim_params_random})
		   \label{eq:Robtest1}
		\end{fig}
   \end{figure}

We observe that the feedback \eqref{eq:backcontr} is robust: it still stabilizes the dynamics at extinction equilibrium if the changes in the  parameters are not too large.

To apply the feedback \eqref{eq:backcontr} we must estimate the number of male and female mosquitoes and the number of eggs.  Some techniques used to measure these parameters are  CDC light traps and BG-Sentinel traps. Based on mosquito behavior, such as their attraction to pheromones or light, these traps use different attractants, such as light, CO2, or human odor, to capture them. To estimate the population size and the ratio of sterile to fertile mosquitoes a common technique is to do Mark-release-recapture (MRR) studies. It consists in marking a subset of the released mosquitoes with a unique identifier and releasing then into the wild. By comparing the number of marked and unmarked mosquitoes captured in the traps, an estimate of the total population size and the ratio of sterile to fertile mosquitoes can be obtained. Some oviposition traps may be used to capture and count the number of eggs laid by female mosquitoes. To take into account the possible difficulty and cost of measuring all the variables ($E, F,M$ and $M_s$)in the field,  in the next sections (\ref{sec:male-mosquitoes} and \ref{sec:fertile-male-mosquitoes}), we propose feedback laws depending on less variables.

\subsection{Feedback laws  depending only on the total number of male mosquitoes}
\label{sec:male-mosquitoes}
Some recent adult traps are able to count automatically the number of male mosquitoes that are captured and, even in a more classic setting, there exist traps that use synthetic versions of female insect pheromones  to attract and capture male insects. This kind of traps  placed at different locations in the field, allow us to determine $M+M_s$ of the target pest population. Our aim in this section is to  build a feedback  linearly depending on $M+M_s$. Consider the closed-loop system
\begin{equation}
\dot{z}= F(z,u(z)),\; z=(E,M,F,M_s)^T\in \mathcal{D}', \label{eq:closed-loop-u(z)-male}
\end{equation}
where
\begin{align}
u(z) = k(M+M_s),\label{eq:feedalpha}
\end{align}
\begin{equation}
\label{eq:defF}
 F(z,u) = \left( \begin{array}{ccc} \beta_E F \left(1-\frac{E}{K}\right) - \left( \nu_E + \delta_E \right) E\\  (1-\nu)\nu_E E - \delta_M M\\  \nu\nu_E E \frac{M}{M+\gamma_s M_s} - \delta_F F\\ u -\delta_sM_s
\end{array}
\right),
\end{equation}
and $k$ is a fixed real number.
Throughout all this section \ref{sec:male-mosquitoes}, we assume that \eqref{R0>1} holds and that
\begin{gather}
\label{alpha<deltas}
k\in [0,\delta_s ).
\end{gather}
The offspring number related to this system  is
             \begin{equation}
                           \;\;\mathcal{R}_1(k):= \frac{(\delta_s-k)\beta_E\nu\nu_E}{\delta_F (\nu_E+\delta_E)(\delta_s-(1-\gamma_s)k)}.
                           \label{offspringR1alpha}
             \end{equation}

             \subsubsection{Equilibria of the closed-loop system}
             Equilibria of the SIT model \eqref{eq:closed-loop-u(z)-male} are obtained by solving the system
             \begin{equation}
                           \left\{
                           \begin{aligned}
                                        & \beta_E F \left(1-\frac{E}{K}\right) - \big( \nu_E + \delta_E \big) E = 0,  \\
                                        & (1-\nu)\nu_E E - \delta_M M =0,  \\
                                        &
                                        \nu\nu_E E \frac{M}{M+\gamma_s M_s} - \delta_F F=0, \\
                                        & k M -(\delta_s-k)Ms =0.
                           \end{aligned}
                           \right.
                           \label{eq:feed}
             \end{equation}
We get either the extinction equilibrium $\textbf{0}$,  i.e.
\begin{gather}
{E} =0,\; \; {M} = 0,\;\;{F} = 0, \;\;{M}_s =  0
\end{gather}
or
\begin{align}
\label{nontrivil equlibrium}
	\begin{aligned}
	{E^*} = K(1 -\frac{1}{\mathcal{R}_1(k)}), \\\; {M^*} = \frac{(1-\nu)\nu_E}{\delta_M}{E^*},\\\;{F^*} = \frac{(\delta_s-k)\nu\nu_E}{\delta_F((\delta_s-k)+\gamma_sk)}{E^*},\\ \;\;{M_s ^*}=  \frac{(1-\nu)\nu_Ek}{(\delta_s-k)\delta_M}{E^*}.
	\end{aligned}
\end{align}
Let us assume in the sequel that
\begin{align}
\label{R1(alpha)<1}
\mathcal{R}_1(k )< 1.
\end{align}
Using \eqref{nontrivil equlibrium} and \eqref{R1(alpha)<1}, one gets $E^*<0$ and therefore the equilibrium given by \eqref{nontrivil equlibrium} is not relevant. In conclusion
the closed-loop system \eqref{eq:closed-loop-u(z)-male} has one and only one equilibrium which is the extinction equilibrium $\textbf{0}$. It is therefore tempting to raise the following conjecture (compare with Theorem~\ref{th-as-without-Ms}).
\begin{conj}
The extinction equilibrium $\textbf{0}$ is globally asymptotically stable in $\mathcal{D}'$ for the closed-loop system \eqref{eq:closed-loop-u(z)-male}.
\end{conj}
We have not been able to prove this conjecture. However,
\begin{enumerate}
\item In section \ref{sec-invariant}, we give a positively invariant set for the closed-loop system \eqref{eq:closed-loop-u(z)-male} in which, as proved in section~\ref{sec-stab-invariant}, $\textbf{0}$ is globally asymptotically stable for \eqref{eq:closed-loop-u(z)-male};
\item In section~\ref{sec-numerical-evidence}, we provide numerical evidence for this conjecture.
\end{enumerate}

\subsubsection{Invariant set of the closed-loop system}
\label{sec-invariant}
 From \eqref{R0>1}, \eqref{alpha<deltas}, and \eqref{R1(alpha)<1}, one gets
\begin{align}
             \frac{\beta_E\nu\nu_E- (\nu_E+\delta_E)\delta_F}{\beta_E\nu\nu_E-(1-\gamma_s) (\nu_E+\delta_E)\delta_F}\delta_s<k<\delta_s.\label{eq:hypotheses-alpha2}
\end{align}
Let us define, with $z=(E,M,F,M_s)^T$,
\begin{align}
\label{defT1}
\mathcal{T}_1 :=\{z\in \mathcal{D}' : \beta_EF(1-\frac{E}{K})\leq  (\nu_E+\delta_E)E\},
\\
\label{defT3}
\mathcal{T}_3 := \{z\in\mathcal{D}': (1-\nu)\nu_EE\leq \delta_MM\},
\end{align}
and, for $\kappa>0$,
\begin{align}
\label{defT2}
\mathcal{T}_2(\kappa) = \{z\in\mathcal{D}' : M\leq \kappa M_s\}.
\end{align}

One has the following theorem.
\begin{theorem}
\label{th-invariant}
Assume that \eqref{eq:hypotheses-alpha2} holds and that
\begin{gather}
\label{kappa<}
\kappa \leq \frac{\gamma_s\delta_F(\nu_E+\delta_E)}{\beta_E\nu\nu_E-\delta_F(\nu_E+\delta_E)},
\\
\label{kappa>}
\kappa \geq \frac{\delta_s-k}{k}.
\end{gather}
Then $\mathcal{M}(\kappa):= \mathcal{T}_1\cap\mathcal{T}_2(\kappa)\cap \mathcal{T}_3 $ is a positively invariant set of the closed-loop system \eqref{eq:closed-loop-u(z)-male}.
\end{theorem}
\begin{remark}
Note that \eqref{eq:hypotheses-alpha2} implies that
\begin{gather}
\label{ineqoverlinekappa}
0<\frac{\delta_s-k}{k}<\frac{\gamma_s\delta_F(\nu_E+\delta_E)}{\beta_E\nu\nu_E-\delta_F(\nu_E+\delta_E)}.
\end{gather}
Hence there are $\kappa>0$ such that both \eqref{kappa<} and \eqref{kappa>} hold.
\end{remark}
\begin{proof}[Proof of Theorem \ref{th-invariant}]
Let us first study the case where one starts with $E=F=0$ : we consider the Filippov solution(s) to the Cauchy problem
 \begin{gather}
 \label{Cauchy-0}
 \dot z =F(z,u(z)),  \; E(0)=0, \,M(0)=M_0,\, F(0)=0, \; M_s(0)=M_{s0},
 \end{gather}
 where $(M_0,M_{s0})^T\in [0,+\infty)^2$ is such that
 \begin{gather}
 \label{ciMetMs}
 M_0\leq \kappa M_{s0}.
 \end{gather}
  From \eqref{eq:feedalpha}, \eqref{eq:defF}, and \eqref{Cauchy-0}, one gets
 \begin{gather}
 \label{E=F=0forallt}
 E(t)=F(t)=0, \; \forall t\geq 0,
 \\
\label{dotM-dotMs}
\dot M=-\delta_M M \text{ and } \dot M_s =k M -(\delta_s-k)M_s.
 \end{gather}
In particular, for every $t\geq 0$, $z(t)\in \mathcal{T}_1\cap \mathcal{T}_3 $. It remains to check that
\begin{gather}
\label{zinT2kappa}
z(t)\in \mathcal{T}_2(\kappa)\; \forall t\geq 0.
\end{gather}
 From \eqref{dotM-dotMs}, one has
\begin{gather}
\label{dsurdtM-kappaMs}
\frac{d}{dt}\left(M-\kappa M_s\right)=-(\delta_M+\kappa k)(M-\kappa M_s)
-\kappa ((1+\kappa)k-\delta_s+\delta_M)M_s.
\end{gather}
 From \eqref{kappa>} one has
\begin{gather}
\label{terminMs<0}
(1+\kappa)k-\delta_s+\delta_M\geq \delta_M.
\end{gather}
Property \eqref{zinT2kappa} readily follows from \eqref{ciMetMs}, \eqref{dsurdtM-kappaMs} and \eqref{terminMs<0}.

Let us now deal with the case where $E+F>0$. Note that, for $z\in \mathcal{M}(\kappa)$, this implies that
\begin{gather}
\label{E>0andM>0}
E>0 \text{ and } M>0.
\end{gather}
Until the end of the proof of Theorem \ref{th-invariant} we assume that $z\in \mathcal{D}'$ and is such that \eqref{E>0andM>0} holds.\\
Let $h_1:\mathcal{D}'\rightarrow \RR$ be defined by

\begin{equation}
 \label{defh1}
h_1(z) := \beta_EF(1-\frac{E}{K})- (\nu_E+\delta_E)E.
\end{equation}
 Its time derivative  along the solution of the closed-loop system
 \eqref{eq:closed-loop-u(z)-male} is
 \begin{multline}
 \dot h_1(z) = \beta_E\nu\nu_EE\frac{M}{M+\gamma_sM_s}(1-\frac{E}{K}) \\-\delta_F\beta_EF(1-\frac{E}{K})-\frac{\beta_E^2F^2}{K}(1-\frac{E}{K}) \\+ \frac{\beta_E(\nu_E+\delta_E)EF}{K}-(\nu_E+\delta_E)\beta_EF(1-\frac{E}{K})+(\nu_E+\delta_E)^2E.
 \end{multline}
 For a set $\Sigma \subset \mathcal{D}'$, let us denote by $\partial \Sigma$ its boundary in $\mathcal{D}'$. On  $\partial\mathcal{T}_1$, $$\beta_EF(1-{E}/{K})=(\nu_E+\delta_E)E$$. Hence
\begin{equation}
\dot h_1(z) = \beta_E\nu\nu_EE\frac{M}{M+\gamma_sM_s}(1-\frac{E}{K}) -\delta_F(\nu_E+\delta_E)E \text{ if } z \in \partial\mathcal{T}_1.
\end{equation}
In particular, using \eqref{kappa<},
\begin{align}
 \dot h_1(z)\leq  -\beta_E\nu\nu_E\frac{M}{M+\gamma_sM_s}\frac{E^2}{K}<0 \text{ if } z \in  \partial\mathcal{T}_1\cap\mathcal{T}_2(\kappa).
\end{align}

Let us now turn to the behavior of the closed-loop system on the $\partial \mathcal{T}_2(\kappa)$.
Let $$h_2:\mathcal{D}'\rightarrow \RR$$ be defined by
 \begin{equation}
 \label{defh2}
h_2(z) := M-\kappa M_s.
 \end{equation}
 Its time derivative  along the solution of the closed-loop system
 \eqref{eq:closed-loop-u(z)-male} is
\begin{gather}
\dot h_2(z)=(1-\nu)\nu_E E - \delta_M M-\kappa\left(k M -\left(\delta_s-k\right) M_s\right),
\end{gather}
which leads to
\begin{gather}
\label{doth2-1}
\dot h_2(z)=(1-\nu)\nu_E E - ((1+\kappa)k-\delta_s+\delta_M)M \text{ if }
z \in \partial \mathcal{T}_2(\kappa).
\end{gather}
 From \eqref{defT3}, \eqref{terminMs<0}, and \eqref{doth2-1}, one gets that
\begin{gather}
\label{doth2-1-new}
\dot h_2(z)\leq 0 \text{ if }
z \in \mathcal{T}_3 \cap \partial \mathcal{T}_2(\kappa).
\end{gather}

 Finally, let us study the behavior of the closed-loop system on the $\partial \mathcal{T}_3$.
Let $$h_3:\mathcal{D}'\rightarrow \RR$$ be defined by
 \begin{equation}
 \label{defh3}
h_3(z) := (1-\nu)\nu_EE - \delta_MM.
 \end{equation}
 Its time derivative  along the solution of the closed-loop system
 \eqref{eq:closed-loop-u(z)-male} is
\begin{gather}
\dot h_3(z)=\beta_E F \left(1-\frac{E}{K}\right) - \big( \nu_E + \delta_E \big) E  - \delta_M \left((1-\nu)\nu_E-\delta_M M\right),
\end{gather}
which leads to
\begin{gather}
\label{doth3-1}
\dot h_3(z)=\beta_E F \left(1-\frac{E}{K}\right) - \big( \nu_E + \delta_E \big) E \text{ if }
z \in \partial \mathcal{T}_3.
\end{gather}
In particular,
\begin{gather}
\label{doth3-2}
\dot h_3(z)\leq -\beta_E \frac{EF}{K} \leq 0\text{ if }
z \in \mathcal{T}_2(\kappa)\cap \partial \mathcal{T}_3.
\end{gather}
This concludes the proof  of Theorem \ref{th-invariant}.
\end{proof}

\subsubsection{Global asymptotic stability result}
\label{sec-stab-invariant} Let
\begin{gather}
\label{defoverliekappa}
\overline \kappa : =\frac{\gamma_s\delta_F(\nu_E+\delta_E)}{\beta_E\nu\nu_E-\delta_F(\nu_E+\delta_E)},
\\
\mathcal{M}:=\mathcal{M}(\overline \kappa).
\end{gather}
Let us recall that, by \eqref{ineqoverlinekappa},  $\overline \kappa$, which clearly satisfies \eqref{kappa<}, satisfies also \eqref{kappa>}. In particular,
by Theorem~\ref{th-invariant}, $\mathcal{M}$ is positively invariant for the closed-loop system \eqref{eq:closed-loop-u(z)-male}.
The main result of this section is the following theorem.
\begin{theorem}\label{eq:thmsys}
 Assume that \eqref{eq:hypotheses-alpha2} holds.  Then $\textbf{0}$ is  globally asymptotically stable for the closed-loop system \eqref{eq:closed-loop-u(z)-male} in $\mathcal{M}$.
 \end{theorem}
 \begin{proof}
  The first step of the proof is the following lemma which shows that Theorem \ref{eq:thmsys} holds with $\mathcal{M}$ replaced by $\mathcal{M}(\kappa)$ provided that \eqref{kappa<} is a strict inequality and that \eqref{kappa>} holds.
 \begin{lemma}
 \label{lemK}
Let us assume that \eqref{kappa>} holds and that
\begin{gather}
\label{kappa<strict}
\kappa < \frac{\gamma_s\delta_F(\nu_E+\delta_E)}{\beta_E\nu\nu_E-\delta_F(\nu_E+\delta_E)}.
\end{gather}
Then $\textbf{0}$ is  globally asymptotically stable for system \eqref{eq:closed-loop-u(z)-male} in $\mathcal{M}(\kappa)$.
 \end{lemma} To prove this lemma we use a Lyapunov approach. Our  Lyapunov function is   $U: \mathcal{D}'\rightarrow\RR_+$, $z\mapsto U(z)$,
\begin{equation}
\label{defU}
 U(z) = \delta_F E +\varepsilon M +\beta_E(1+\varepsilon)F+\varepsilon^2M_s,
\end{equation}
where $\varepsilon\in (0,1]$ is a constant which will be chosen later on. One has
\begin{gather}
U \text{ is of class $\mathcal{C}^1$},
\\
\label{U>0}
U(z)>U(\textbf{0}) =0,\;\forall z \in \mathcal{D}'\setminus\{\textbf{0}\},
\\
\label{Uzgrandsizgrand}
U(z)\rightarrow +\infty \text{ as } |z|\rightarrow +\infty \text{ with } z\in \mathcal{D}'.
\end{gather}
Let us assume for the time being that
\begin{gather}
\label{M+MSnot0}
M+M_s\not =0.
\end{gather}
Then, the time derivative  of $U$ along the solution of the closed-loop system
 \eqref{eq:closed-loop-u(z)-male} is
\begin{multline}
\label{dotU-1}
\dot U (z)= \delta_F \left(\beta_E F \left(1-\frac{E}{K}\right) - \left( \nu_E + \delta_E \right) E\right)
+\varepsilon \left((1-\nu)\nu_E E - \delta_M M\right)
\\+\beta_E(1+\varepsilon)\left(\nu\nu_E E \frac{M}{M+\gamma_s M_s} - \delta_F F\right)
+\varepsilon^2\left(k M -(\delta_s-k)M_s\right).
\end{multline}
In particular,
\begin{multline}
\label{dotU-2}
\dot U (z)\leq  -\varepsilon\delta_F\beta_E F  \\- \left(\left( \nu_E + \delta_E \right)-\varepsilon (1-\nu)\nu_E-\beta_E(1+\varepsilon)\nu\nu_E \frac{\kappa}{\kappa +\gamma_s} \right)E
\\- \varepsilon \left(\delta_M-\varepsilon k\right)M -\varepsilon^2(\delta_s-k)M_s \text{ if } z\in \mathcal{M}(\kappa).
\end{multline}
Let us now point out that \eqref{kappa<strict} implies that
\begin{gather}
\label{terminE>0}
\beta_E\nu\nu_E \frac{\kappa}{\kappa +\gamma_s} <\left( \nu_E + \delta_E \right).
\end{gather}
 From \eqref{dotU-2} and \eqref{terminE>0} one gets that for $\varepsilon >0$ small enough there exists $c(\varepsilon)>0$ independent of $z\in \mathcal{M}(\kappa)$ such that
 \begin{gather}
\label{dotU<-cU}
\dot U (z)\leq -c(\varepsilon) U(z) \text{ if } z\in \mathcal{M}(\kappa).
\end{gather}
It remains to remove assumption \eqref{M+MSnot0}. Let $$t\mapsto z(t)=(E(t),M(t),F(t),M_s(t))^T$$ be a Filippov solution of the closed loop system for the initial condition $z(0)=(E_0,M_0,F_0,M_{s0})^T\in \mathcal{M}(\kappa)$. We observe that if $(E_0,F_0)=(0,0)$, then $z(0)\in \mathcal{M}(\kappa)$ implies that $M_0>0$, from which one gets that $M(t)>0$ for every $t\geq 0$. Hence \eqref{dotU<-cU} holds for every $t\geq 0$. While, if $(E_0,F_0)\not= (0,0)$, then $M(t)>0$ for every $t>0$. In particular, one still has \eqref{dotU-2} and therefore \eqref{dotU<-cU} for every $t> 0$. Hence,
\begin{gather}
U(z(t))\leq e^{-c(\varepsilon)t}U(z(0)),\; \forall t\geq 0,
\end{gather}
which, together with \eqref{U>0} and \eqref {Uzgrandsizgrand}, concludes the proof of Lemma~\ref{lemK}.

Let us now deduce from Lemma~\ref{lemK} that
\begin{gather}
\label{global-attractor-in-M}
\text{$\textbf{0}$ is a global attractor for the closed-loop system \eqref{eq:closed-loop-u(z)-male} in $\mathcal{M}$.}
\end{gather}
Let $z(t)=(E(t),M(t),F(t),M_s(t))^T$ be a Filippov solution of the closed-loop system \eqref{eq:closed-loop-u(z)-male} for the initial condition $z(0)=(E_0,M_0,F_0,M_{s0})^T\in \mathcal{M}(\kappa)$. If  $(E_0,F_0)=(0,0)$ then one has \eqref{E=F=0forallt} and \eqref{dotM-dotMs} which leads to $z(t)\rightarrow \textbf{0}$ as $t\rightarrow +\infty$ (note that, by \eqref{ineqoverlinekappa}, $\delta_s-k>0$).
Let $h_2:\mathcal{D}'\rightarrow \RR$ be defined by
 \begin{equation}
 \label{defoverlineh2}
\overline h_2(z) := M-\overline \kappa M_s.
 \end{equation}
 Note that, if for some $t_0\geq 0,$ $\overline h_2(z(t))<0$, then there exists $\kappa>0$ satisfying \eqref{kappa>} and \eqref{kappa<strict} such that $z(t_0)\in \mathcal{M}(\kappa)$. By Lemma~\ref{lemK} one then has $z(t)\rightarrow \textbf{0}$ as $t\rightarrow +\infty$. If there is no such $t_0 ,$ then
\begin{gather}
\label{h2=0}
\overline h_2(z(t))=0 \text{ for every $t\geq 0$. }
\end{gather}
 From \eqref{doth2-1} with $\kappa =\overline \kappa$, \eqref{defoverliekappa}, \eqref{defoverlineh2}, and  \eqref{h2=0}, one gets that
\begin{gather}
\label{h3=0}
 h_3(z(t))=0 \text{ for every $t\geq 0$, }
\end{gather}
which together with \eqref{doth3-2} implies that
\begin{gather}
\label{EF=0}
E(t)F(t)=0 \text{ for every $t\geq 0$. }
\end{gather}
Since $z(t)\in \mathcal{T}_1$, \eqref{defT1} and \eqref{EF=0} imply that
\begin{gather}
\label{F=0}
 F(t)=0 \text{ for every $t\geq 0$. }
\end{gather}
Then, if for some $t_0\geq 0$, $E(t_0)=0$, one has $(E(t_0),F(t_0))=(0,0)$, which, as already pointed out above, implies that $z(t)\rightarrow \textbf{0}$ as $t\rightarrow +\infty$. It remains to handle the case where
\begin{gather}
\label{E>0}
E(t)>0 \text{ for every $t\geq 0$. }
\end{gather}
In particular, since $z(t)\in \mathcal{T}_3$, one has, using \eqref{defT3},
\begin{gather}
M(t)>0 \text{ for every $t\geq 0$. }
\end{gather}
Then, differentiating \eqref{F=0} with respect to time and using \eqref{eq:closed-loop-u(z)-male} and \eqref{eq:defF}, one gets
\begin{gather}
E(t)=0 \text{ for every $t\geq 0$,}
\end{gather}
which leads to a contradiction with \eqref{E>0}.
 This concludes the proof of \eqref{global-attractor-in-M}.

In order to conclude the proof of Theorem~\ref{eq:thmsys} it just remains to check that
\begin{gather}
\label{stable-in-M}
\text{$\textbf{0}$ is stable  for the closed-loop system \eqref{eq:closed-loop-u(z)-male} in $\mathcal{M}$.}
\end{gather}
 For that, let $\overline U: \mathcal{D}'\rightarrow\RR_+$, $z\mapsto \overline U(z)$, be defined by
\begin{equation}
\overline  U(z) = \delta_F E +\beta_EF,
\end{equation}
which corresponds to the definition of $U$ given in \eqref{defU} with $\varepsilon=0$. Let $$z(t)=(E(t),M(t),F(t),M_s(t))^T$$ be a Filippov solution of the closed loop system for the initial condition $z(0)=(E_0,M_0,F_0,M_{s0})^T\in \mathcal{M}$. As above, we may restrict our attention to the case where
\begin{gather}
\label{E(t)>0}
E(t)>0 \text{ for every $t>0$.}
\end{gather}
Let us recall that since $z(t)\in \mathcal{M}\subset \mathcal{T}_3$, \eqref{defT3}, and \eqref{E(t)>0} imply that
\begin{gather}
\label{M(t)>0}
M(t)>0 \text{ for every $t\geq 0$.}
\end{gather}
Then, $\overline U(z(t))$ can be differentiated with respect to time and one has, by \eqref{dotU-2} with $\varepsilon =0$ and $\kappa =\overline \kappa$, and \eqref{defoverliekappa},
\begin{gather}
\label{dotoverlineU}
\dot {\overline U }(z(t))\leq 0,
\end{gather}
which shows that
\begin{gather}
\label{estimate-E-F}
E(t)+F(t)\leq \frac{\max\{\delta_E,\delta_F\}}{\min\{\delta_E,\delta_F\}}\left(E(0)+F(0)\right), \text{ for every $t\geq 0$.}
\end{gather}
It remains to estimate $M(t)$ and $M_s(t)$. Using $z(t)\in \mathcal{T}_2(\overline \kappa )$ and \eqref{defT2}, one already has
\begin{gather}
\label{M<overlinekappaMs}
M(t)\leq \overline \kappa M_s(t)  \text{ for every $t\geq 0$.}
\end{gather}
Using \eqref{eq:closed-loop-u(z)-male}, \eqref{eq:feedalpha}, \eqref{eq:defF}, \eqref{ineqoverlinekappa}, \eqref{defoverliekappa}, and \eqref{M<overlinekappaMs}, one has
\begin{gather}
\label{M<overlinekappaMs-new}
\dot M_s(t)\leq \left(k \overline \kappa +k -\delta_s\right)M_s(t) \leq 0 \text{ for every $t\geq 0$.}
\end{gather}
In particular, using also \eqref{M<overlinekappaMs},
\begin{gather}
M_s(t)\leq M_s(0) \text{ and } M(t)\leq \overline \kappa M_s(0) \text{ for every $t\geq 0$.}
\end{gather}
This concludes the proof of \eqref{stable-in-M} and, therefore, of Theorem~\ref{eq:thmsys}.
 \end{proof}

\subsubsection{Numerical simulations}
\label{sec-numerical-evidence}
In this section, we will show  numerical simulations of the dynamics when we apply  feedback \eqref{eq:feedalpha}.   We  fix $z_0 = (21910,5587, 13419,0)\notin \mathcal{M}$. We now compute  condition \eqref{eq:hypotheses-alpha2} according to the parameter set in the table \ref{eq:tableparametre}. This gives  $0.11843<k<0.12$. We take $k=0.119$.  The following figures show the evolution of the states when  condition \eqref{eq:hypotheses-alpha2} holds.

\begin{figure}[H]
	\centering
	\begin{subfigure}[H]{0.45\textwidth}
		\centering
		\includegraphics[width=\textwidth]{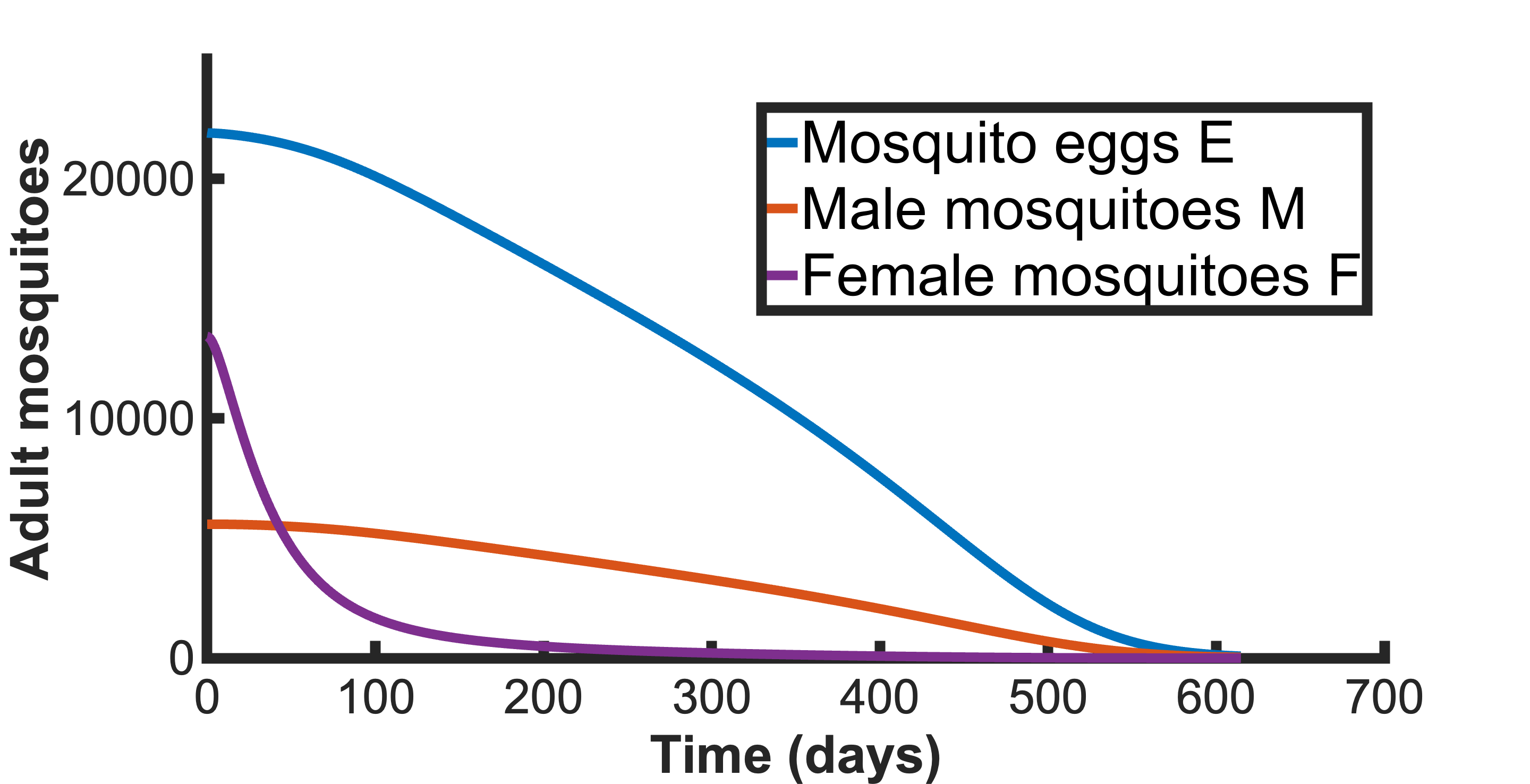}
		\caption{Plot of $E, M$ and $F$}
		\label{fig:evolutionEMF-2}
	\end{subfigure}
	\hfill
	\begin{subfigure}[H]{0.45\textwidth}
		\centering
		\includegraphics[width=\textwidth]{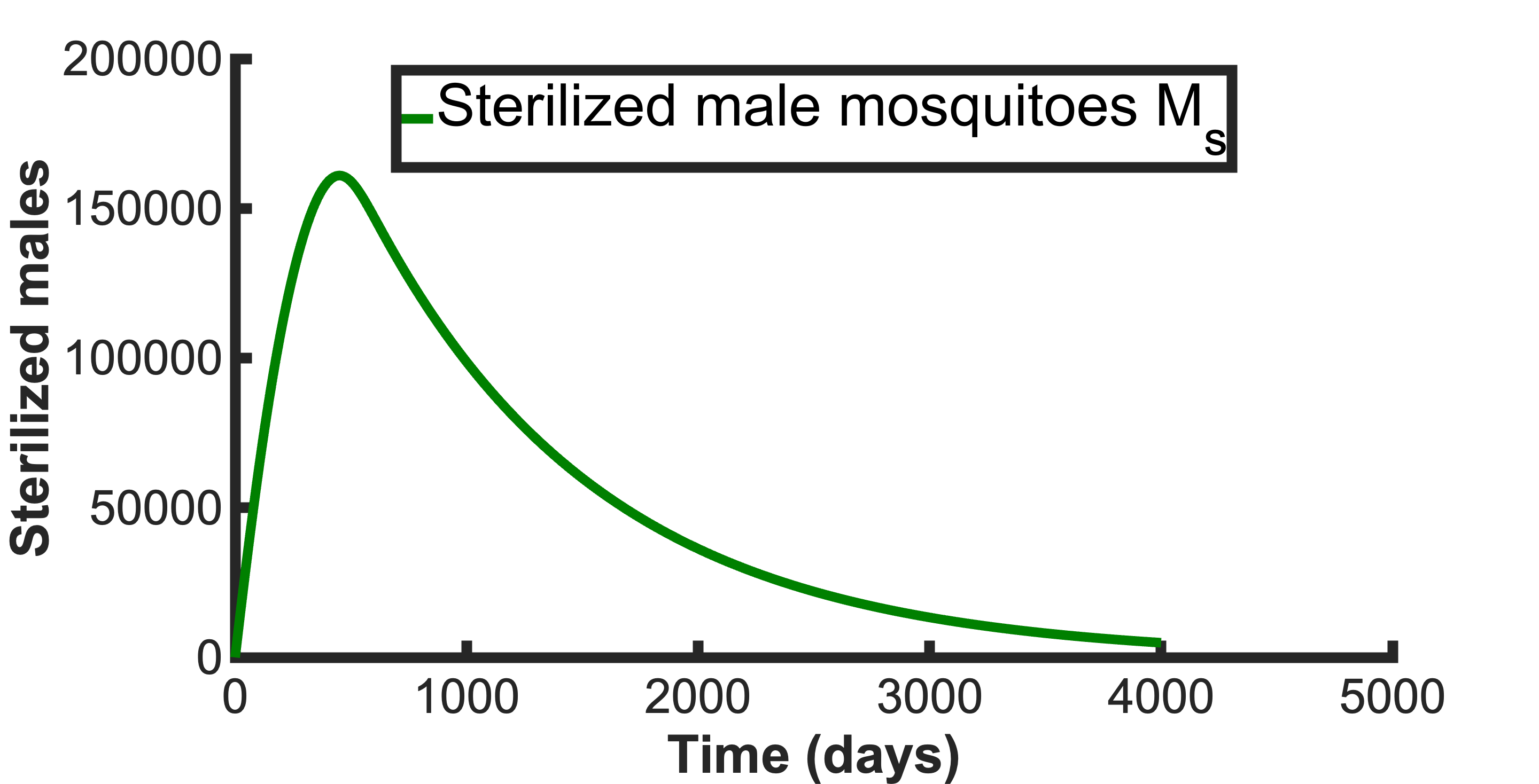}
		\caption{Plot of $M_s$}
		\label{fig:EvolutionMs-2}
	\end{subfigure}
	\hfill
	\begin{subfigure}[H]{0.5\textwidth}
		\centering
		\includegraphics[width=\textwidth]{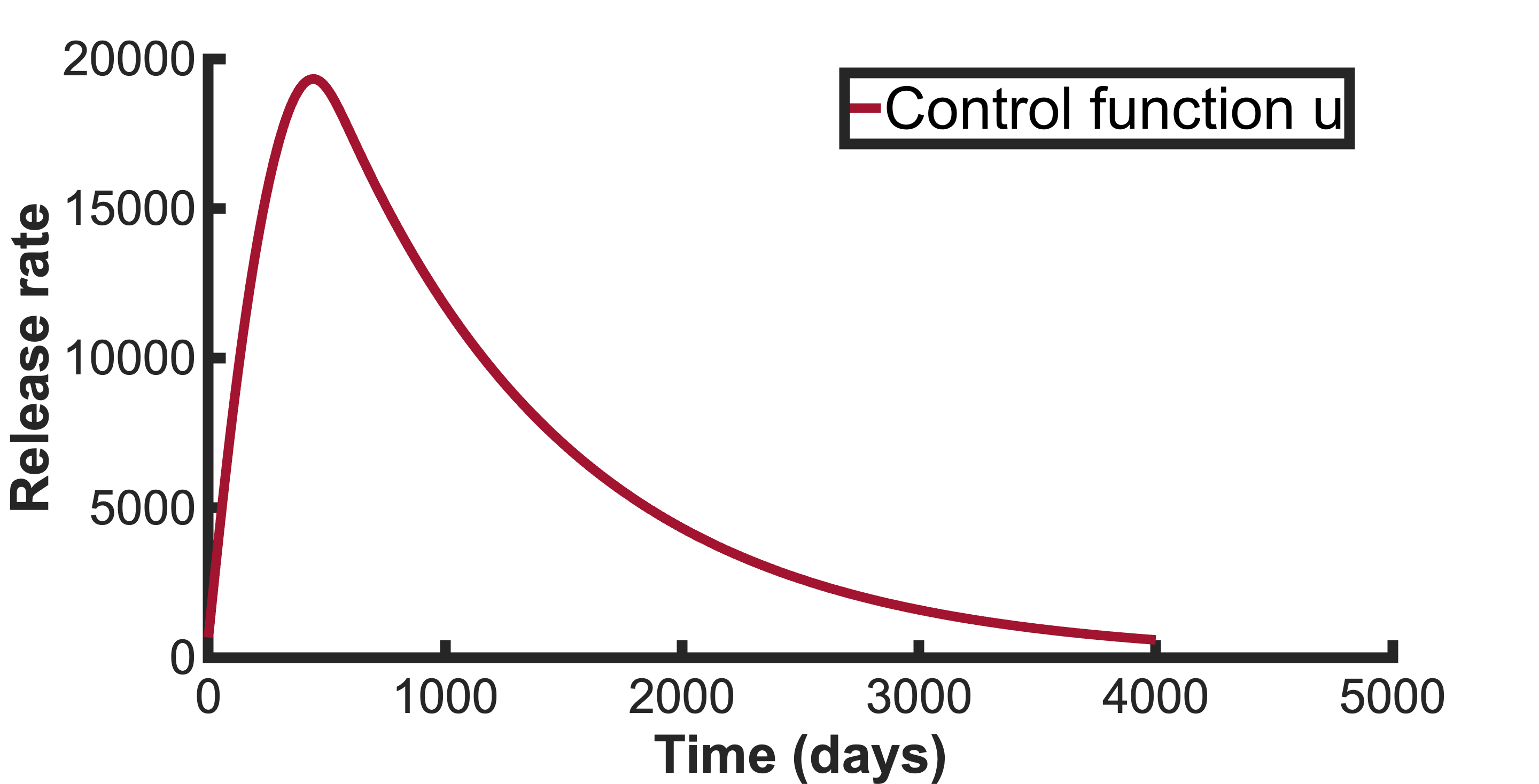}
		\caption{Plot of the control $u$}
		\label{fig:five over x-2}
	\end{subfigure}
	\begin{fig}
		(a): Plot of  $E,M$ and $F$ for system \eqref{eq:S1E1}-\eqref{eq:S1E4} when applying  feedback \eqref{eq:feedalpha}. with  the initial condition $z_0\notin \mathcal{M}$ and final time $T=800$. (b): Plot of  $M_s$ for final time $T=4000$ when we apply the backstepping feedback \eqref{eq:feedalpha}. (c): Plot of the feedback control function \eqref{eq:feedalpha}.
	\end{fig}
	\label{fig:simulation6}
	
\end{figure}

\begin{remark}
	We observe that  the convergence time of the states  $E,M$ and $F$ is longer than when we applied the backstepping feedback control \eqref{eq:backcontr}.
	In this case, with $t_f = 700$ days,
	\begin{align}
		\int_{0}^{t_f} u(t)\; dt \approx  17,91 \;\mbox{millions}.
	\end{align}
\end{remark}
	We take several initial conditions randomly and plot the resulting dynamics in figure \ref{conjecturedynamics},
	\begin{figure}[H]
		\centering
		\includegraphics[width=0.5\textwidth]{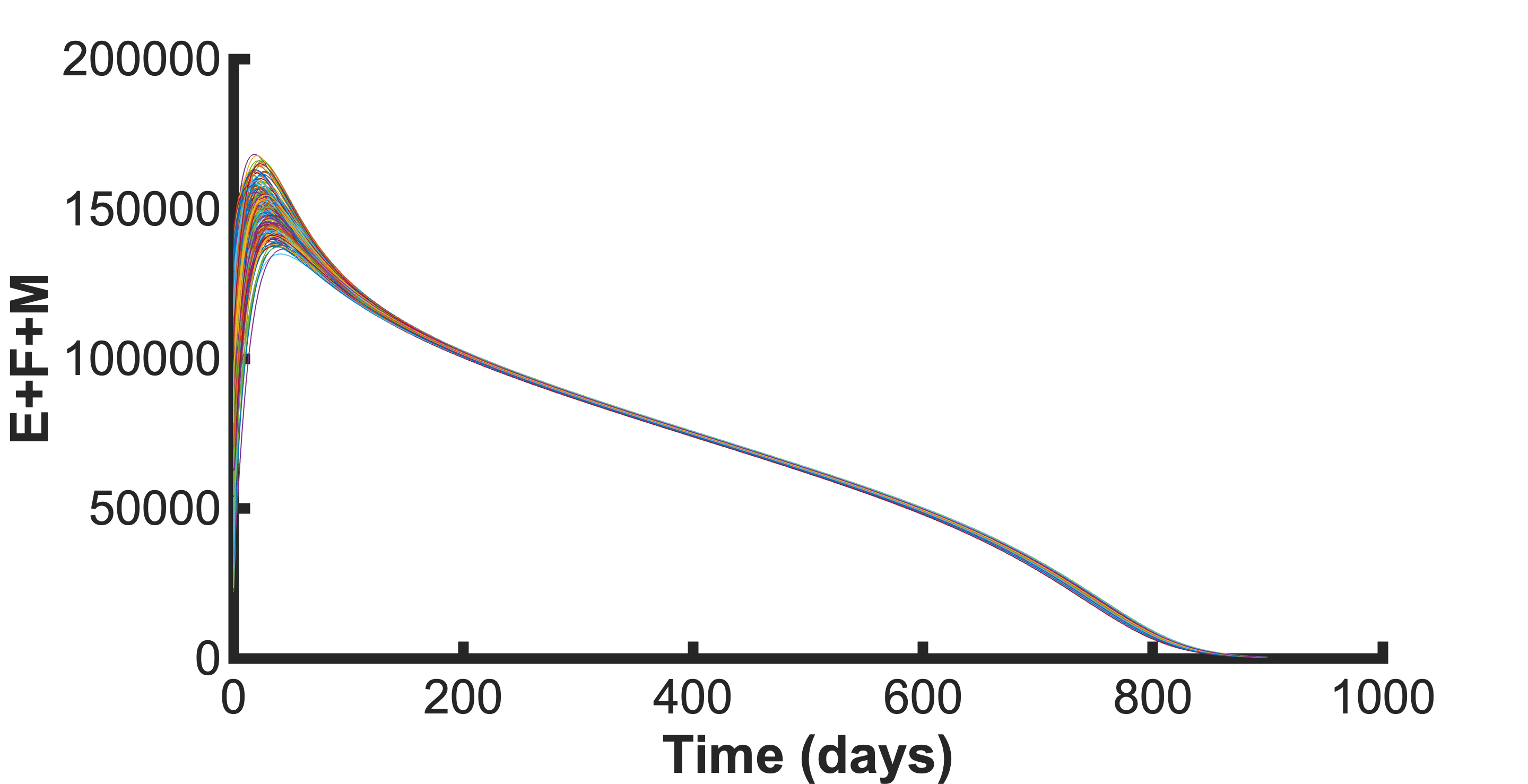}
		 \begin{fig}
   	 Plot of $\norm{x(x_0,t)}_1$ when applying the feedback \eqref{eq:feedalpha} with several randomly chosen initial conditions $x_0$.
   	 \label{conjecturedynamics}
   \end{fig}
	\end{figure}

\subsubsection{Robustness test}
 To analyze the robustness of our feedback against variations of the parameters, we carry out some variation of the parameters (new values) in table \ref{RobTest}. The results are summarized in  table \ref{RobTest}.
\begin{table}[ht]
	\begin{center}
		\setlength{\tabcolsep}{0.01cm}
		\begin{tabular}{ | m{3cm} |m{3cm}| m{7cm} | }
			\hline
			Old parameters  & New Parameters &  Simulation
			\\ \hline
			\begin{itemize}
				\item $\nu_E$= 0.05
				\item $\delta_E$= 0.03
				\item $\delta_F$= 0.04
				\item $\delta_M$ = 0.1
				\item $\delta_s$ = 0.12
				\item $\beta_E$ =8
			\end{itemize}
	&
	\begin{itemize}
		\item $\nu_E$= 0.08
		\item $\delta_E$= 0.046
		\item $\delta_F$= 0.033
		\item $\delta_M$ = 0.12
		\item $\delta_s$ = 0.139
		\item $\beta_E$ =11
	\end{itemize}
	&
	\begin{itemize}
		
		\item Plot of $E, M$ and $F$
		
	\includegraphics[width=0.3\textwidth]{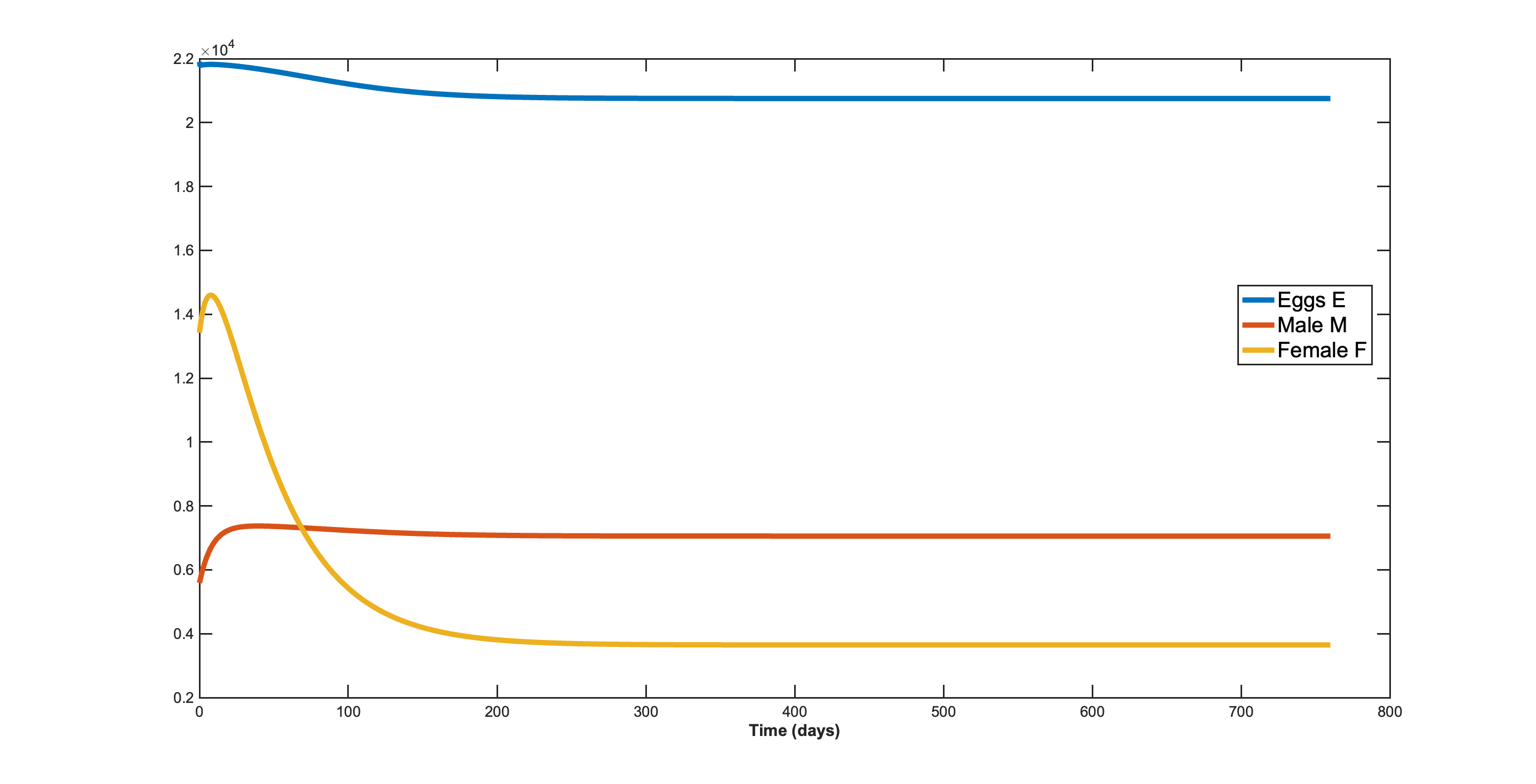}
	\end{itemize}
	\\ \hline
\end{tabular}
\end{center}
\caption{Robustness test}
\label{RobTest}
\end{table}
We observe that very small perturbations of the parameters destabilize the origin.

\subsection{Feedback laws depending only on wild male mosquitoes}
\label{sec:fertile-male-mosquitoes}
In the application of the technique it might also be possible to estimate only fertile males. For instance, in MRR experiments, sterile mosquitoes are identified by the presence of a marker, such as a dye or a fluorescent protein, which has been applied before their release (although, at present, it is not always easy to do this for all the mosquitoes released in field interventions). Nevertheless, since the technology is evolving very fast, it is possible that it can become standard practice in the near future (for instance, we recall that PCR analysis of the captured mosquitoes is already currently used thanks to genetic bar-coding). Thus, it is interesting to set up the mathematical techniques to deal with this situation. Therefore, we consider in this section the case where the feedback depends only on the state $M$. Consider the closed-loop system
\begin{equation}
	\dot{z}= F(z,u(z)),\; z=(E,M,F,M_s)^T\in \mathcal{D}',\label{eq:cloose loop2}\\
\end{equation}
where
\begin{align}
	u(z) = \lambda M\label{eq:feedlambda}
\end{align}
and
\begin{equation}
	F(z,u(z)) = \left( \begin{array}{ccc} \beta_E F \left(1-\frac{E}{K}\right) - \big( \nu_E + \delta_E \big) E\\  (1-\nu)\nu_E E - \delta_M M\\  \nu\nu_E E \frac{M}{M+\gamma_s M_s} - \delta_F F\\  \lambda M -\delta_sM_s
	\end{array}
	\right),	
\end{equation}
The offspring number related to this system  is
\begin{equation}
	\;\;\mathcal{R}_2( \lambda):= \frac{\delta_s\beta_E\nu\nu_E}{\delta_F (\nu_E+\delta_E)(\delta_s+\gamma_s \lambda)}.
	\label{offspringR2lambda}
\end{equation}
We assume that
\begin{align}
\label{R2lambda<1}
\mathcal{R}_2( \lambda)<1.
\end{align}
Note that this inequality is equivalent to
\begin{align}
	\lambda>\frac{(\beta_E\nu\nu_E- (\nu_E+\delta_E)\delta_F)\delta_s}{\gamma_s (\nu_E+\delta_E)\delta_F}.\label{eq:hypotheses lambda}
\end{align}
Let us point out that the closed-loop system \eqref{eq:cloose loop2} is exactly the closed-loop system \eqref{eq:closed-loop-u(z)-male}
 if one performs the following change of variables (with natural notations):
\begin{equation}
\label{chgt-var}
k^{\text{\eqref{eq:closed-loop-u(z)-male}}}=\lambda^{\text{\eqref{eq:cloose loop2}}} \text{ and }
\delta_s^{\text{\eqref{eq:closed-loop-u(z)-male}}}=\delta_s^{\text{\eqref{eq:cloose loop2}}}+\lambda^{\text{\eqref{eq:cloose loop2}}}.
\end{equation}
Hence  Theorem~\ref{th-invariant} and Theorem~\ref{eq:thmsys} lead to the following theorem.
\begin{theorem}\label{eq:thmsys-2}
 Assume that \eqref{R0>1} and \eqref{eq:hypotheses lambda} hold.  Then $\mathcal{M}$ is positively invariant for the closed-loop system \eqref{eq:cloose loop2} and $\textbf{0}$ is  globally asymptotically stable for the closed-loop system \eqref{eq:cloose loop2} in $\mathcal{M}$.
 \end{theorem}

\subsubsection{Numerical simulations}
In this section, we  present the numerical evolution of the states  when we apply  feedback \eqref{eq:feedalpha}. We fix as initial condition $z_0 = (21910,5587, 13419,0)\notin \mathcal{M}$ and  $K=22200 \text{ ha}^{-1}$
We now compute  condition \eqref{eq:hypotheses-alpha2} according to the parameters set in  table \ref{eq:tableparametre}. This gives  $\lambda>9.06$. We take for the simulation $\lambda = 22$.

\begin{figure}[H]
	\centering
	\begin{subfigure}[H]{0.45\textwidth}
		\centering
		\includegraphics[width=\textwidth]{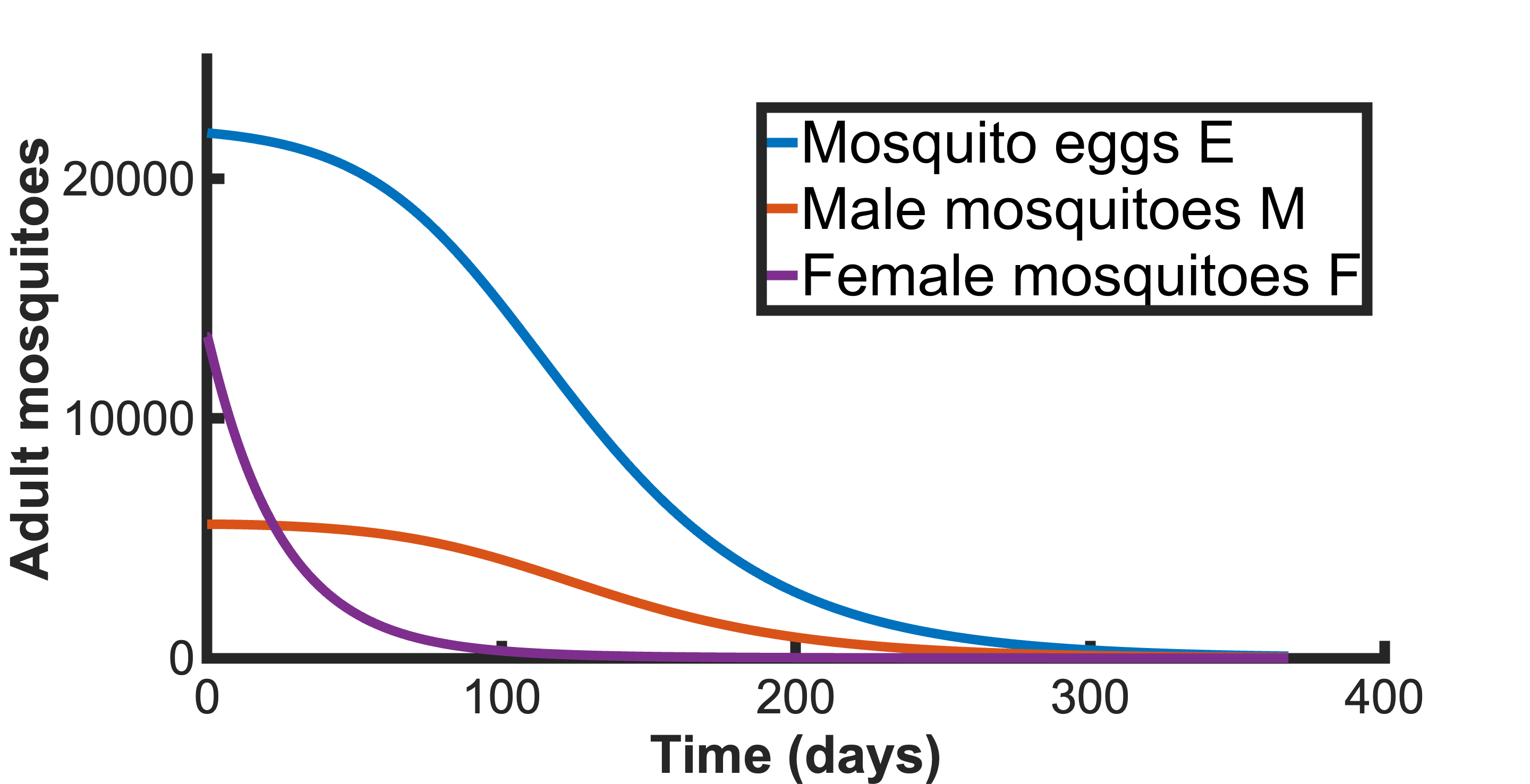}
		\caption{Plot of $E, M$ and $F$}
		\label{fig:evolutionEMF-3}
	\end{subfigure}
	\hfill
	\begin{subfigure}[H]{0.45\textwidth}
		\centering
		\includegraphics[width=\textwidth]{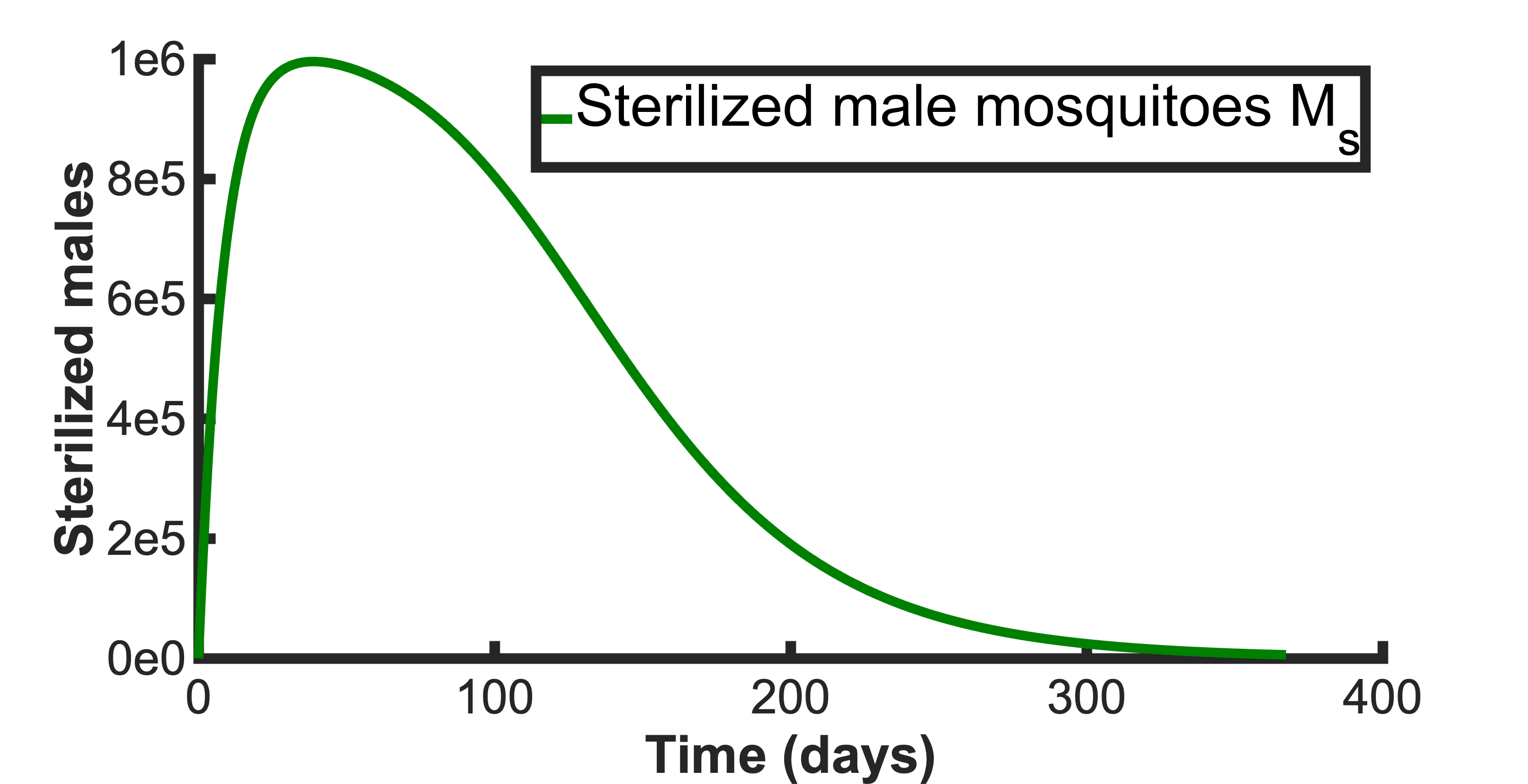}
		\caption{Plot of $M_s$}
		\label{fig:EvolutionMs-3}
	\end{subfigure}
	\hfill
	\begin{subfigure}[H]{0.5\textwidth}
		\centering
		\includegraphics[width=\textwidth]{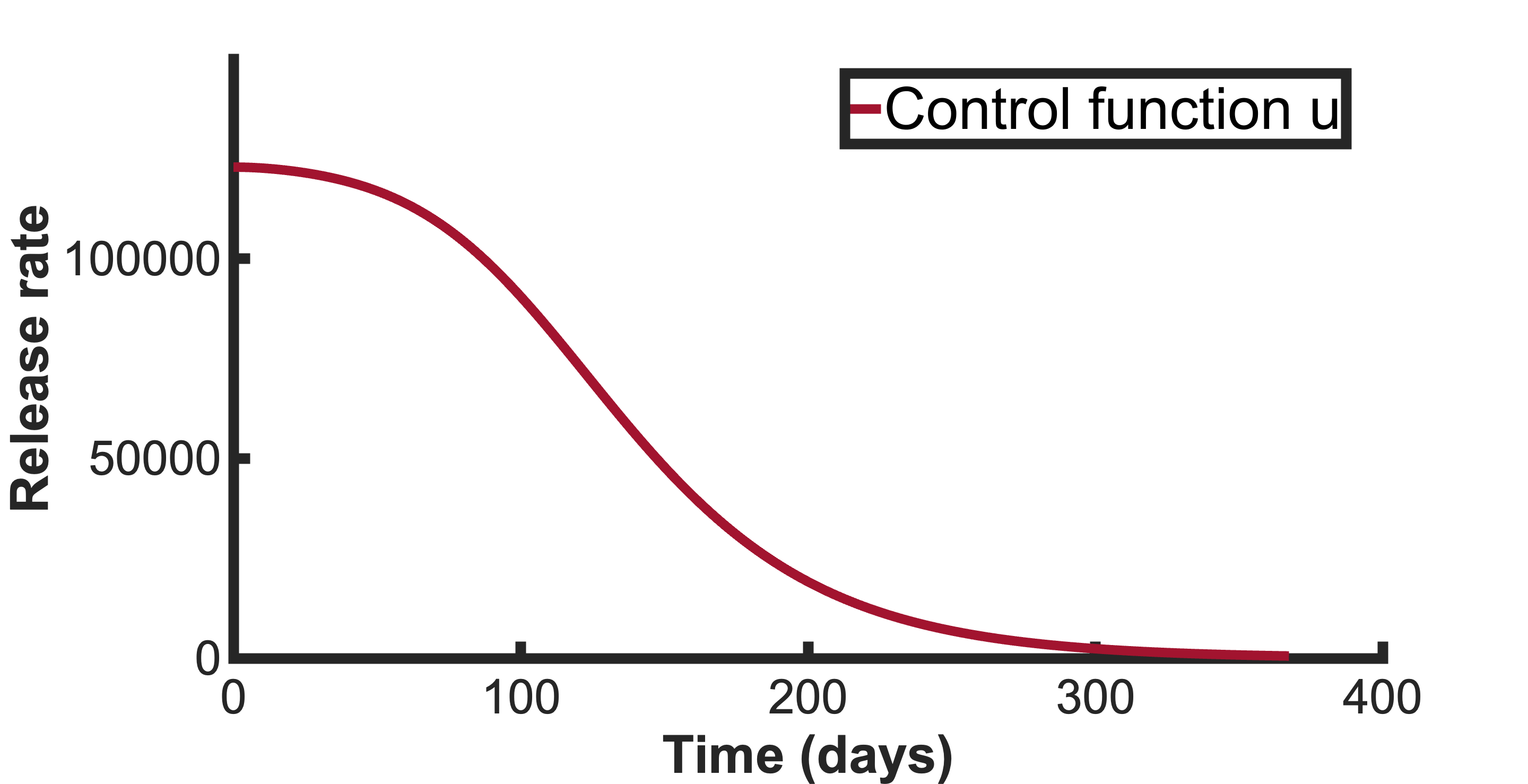}
		\caption{Plot of the control $u$}
		\label{fig:five over x-3}
	\end{subfigure}
	\begin{fig}
		(a): The results of the simulation $E,M$ and $F$ for system \eqref{eq:S1E1}-\eqref{eq:S1E4} when applying the feedback \eqref{eq:feedlambda} with  the initial condition $z_0\notin \mathcal{M}$ for final time $T=400$ and $\lambda=22$. (b): Plot of  $M_s$ for final time $T=400$. (c): Plot of  the control function \eqref{eq:feedlambda}.
	\end{fig}
	\label{fig:simulation4}
\end{figure}

\begin{remark}
	Notice that  with $t_f = 400$ days,
	\begin{align}
		\int_{0}^{t_f} u(t)\; dt \approx  17.28\;\mbox{millions}.
	\end{align}

\end{remark}

In figure \ref{eq:conject} we take several initial conditions randomly for $\lambda = 22$ .
\begin{figure}[ht]
\centering
\includegraphics[width=0.7\textwidth]{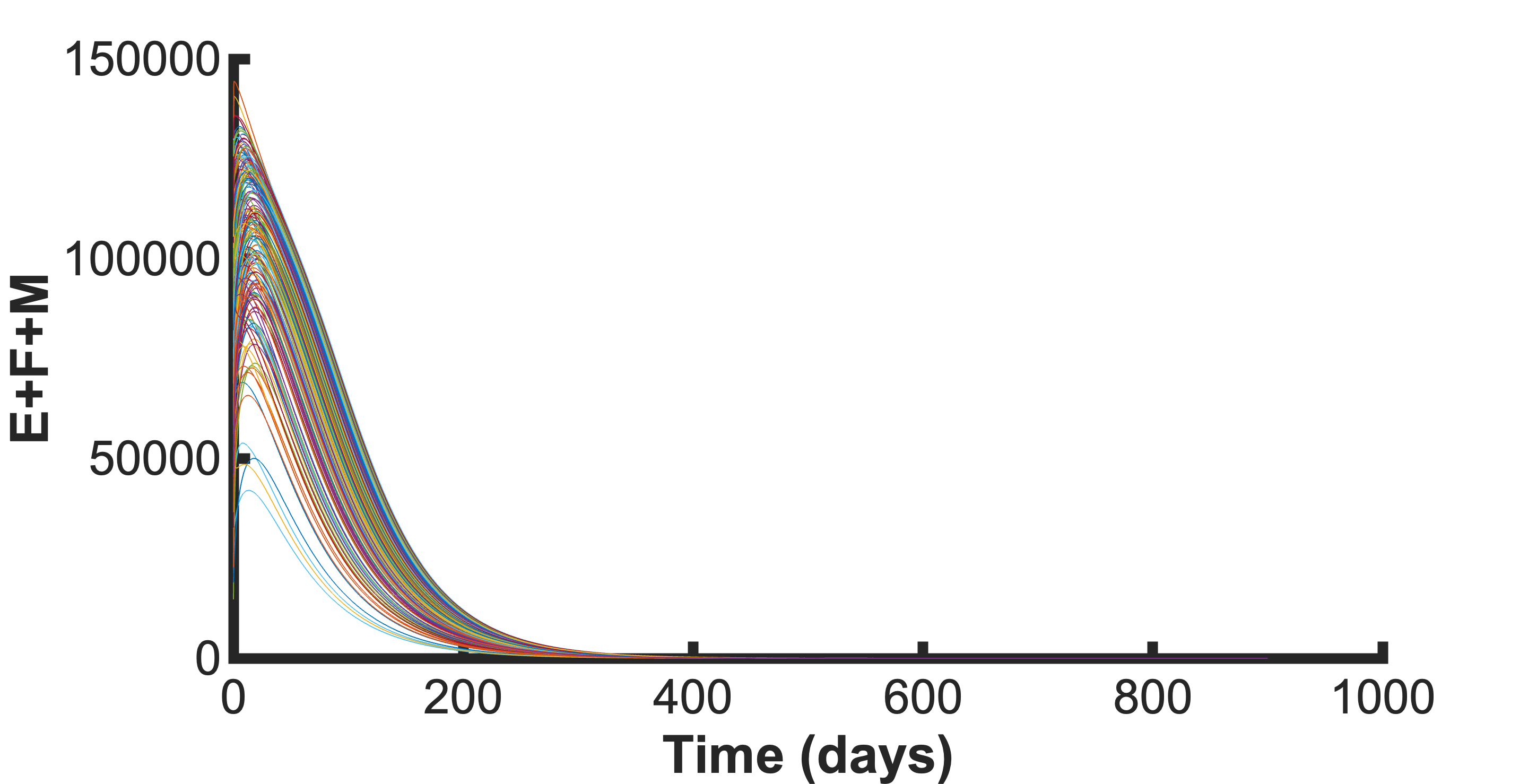}
\begin{fig}
\label{eq:conject}
Plot of $\norm{x(x_0,t)}_1$ when applying the feedback \eqref{eq:feedlambda} with several randomly chosen initial conditions $x_0$.
\end{fig}
		
\end{figure}

\subsubsection{Robustness test}
We test the robustness using the same protocol as in section \ref{see:robusttessforback}. Figure \ref{eq:Robtest-new} shows the results for 200 randomly chosen initial conditions in $[0, 10K]^4$.

\begin{figure}[H]
\centering
\includegraphics[width=0.8\textwidth]{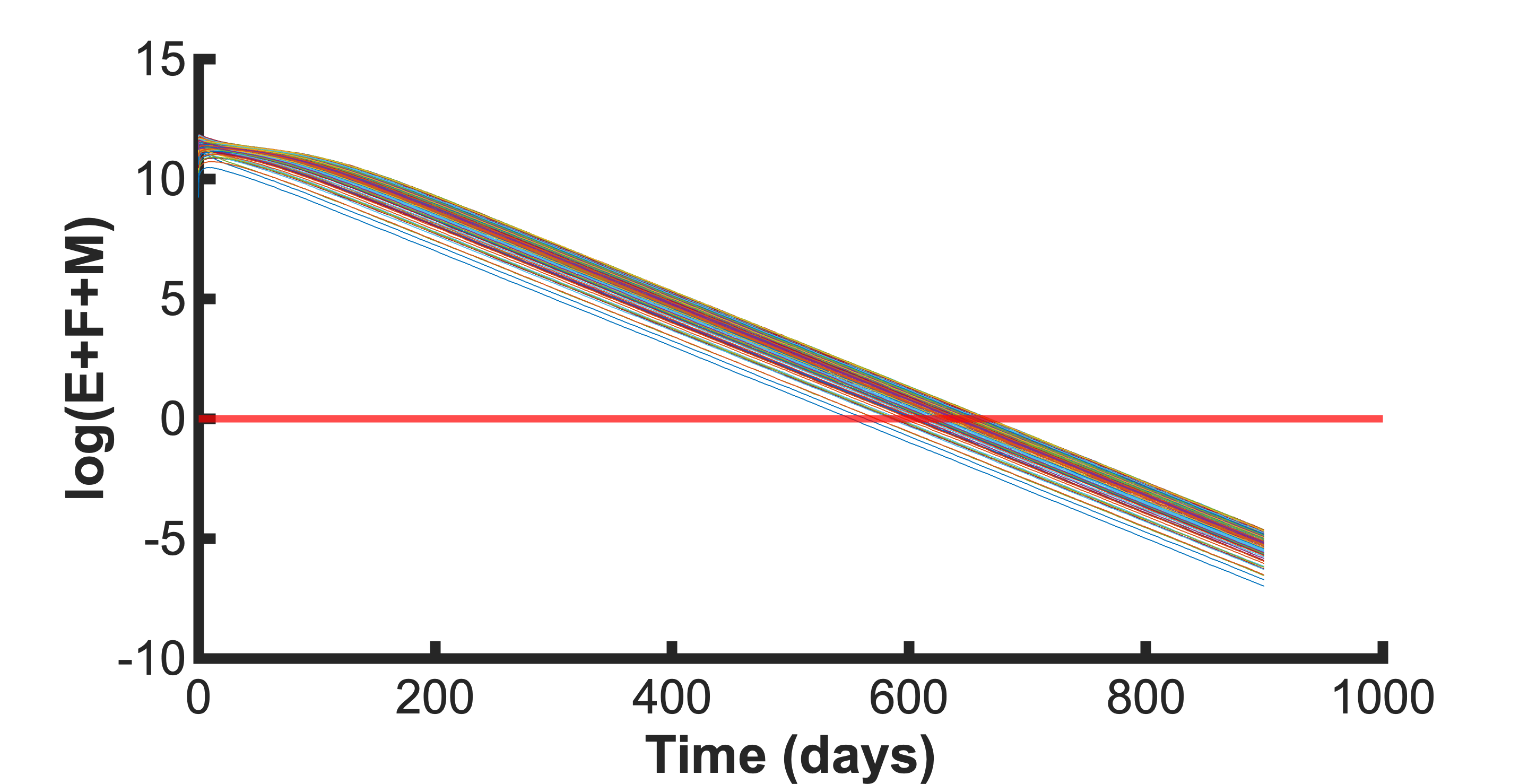}
\begin{fig}
Robustness test when applying the feedback law \eqref{eq:feedlambda} with $\lambda = 22$.
\label{eq:Robtest-new}	
\end{fig}
\end{figure}

We observe that feedback \eqref{eq:feedlambda} is robust  with respect to changes of  parameters: for rather large perturbations on the parameters it stills globally stabilizes the dynamics at the extinction equilibrium.
\section{Comparison of the feedback laws}\label{see:Comparativesect}
	In this section, we use numerical
simulations to carry out a comparative
 study of the feedback  control \eqref{eq:backcontr} and \eqref{eq:feedlambda}.
 We consider that the environmental capacity
	$K=22200 \text{ ha}^{-1}$ and that the initial condition is the persistence  equilibrium.
 Our comparison criteria are the intervention
 time and the control cost obtained when applying
  the different feedback laws. The results are presented
   in Tables \ref{tab:lambdacontrol} and \ref{tab:backst} where the intervention is presented until  $E = \frac{K}{100}$. $\lambda$ is used to regulate the
    control feedback law \eqref{eq:feedlambda}  while the regulation parameters for the
	backstepping control  are $\alpha$,  $\theta$  and $\beta_s=1\text{ Day}^{-1}$.

 Table \ref{tab:lambdacontrol}
	shows the intervention time and control cost
	for different values of  $\lambda$. In table \ref{tab:backst},  we fix  $\alpha = 80$ and
	 present the results obtained for different values of  $\theta$.
	Note that since $\alpha$ and $\theta$ are regulatory values
	for control \eqref{eq:backcontr}, a study can be carried out to find
	their optimal values  in order to have a better
	value of the control \eqref{eq:backcontr} presented here.

	\begin{table}[ht]
	       \tiny
		\centering
		\setlength{\tabcolsep}{0.03cm}
		\renewcommand{\arraystretch}{1.5}
                  \begin{tabular}{|l|c|c|c|c|c|c|c|c|c|c|c|c|c|c|}
			\hline
			\multicolumn{15}{|c|}{$u_{\lambda}$ Intervention } \\
			\hline
			$\lambda$ &9.06 & 10 & 11 & 12 & 13 & 14 & 15 &16 & 17 & 18 & 19 & 20 & 21& 22\\
			\hline
			$T_{\lambda}$ (Day$^{-1}$)  &667& 477& 424& 390& 367& 350& 336& 326& 318& 311& 305& 300& 295& 291\\

			\hline
			$(\int_{0}^{T} u_{\lambda}) $ &8.24e6& 8.61e6&9.14e6&9.72e6&1.03e7&1.09e7&1.16e7& 1.22e7&1.29e7&1.35e7&1.42e7&1.48e7&1.55e7 &1.61e7\\
			\hline
		\end{tabular}
			\caption{Intervention time and  control cost for different values of $\lambda$}
			\label{tab:lambdacontrol}
	\end{table}
	\begin{table}[ht]
	         \tiny
		\centering
		\setlength{\tabcolsep}{0.05cm}
		\renewcommand{\arraystretch}{1.5}
		\begin{tabular}{|l|c|c|c|c|c|c|c|c|c|c|c|c|c|c|c|c|c|c|c|c|c|c|c|}
			\hline
			\multicolumn{14}{|c|}{$u_{\theta}$ Intervention } \\
			\hline
			$\theta$ & 100 & 110 & 120 & 130 & 140 &150 &160& 170 & 180 & 190 & 200 & 210& 220\\
			\hline
			$T_{\theta}$ (Day$^{-1}$) &484& 445& 417& 396& 379& 366& 355& 345& 338& 331& 325& 319& 315\\
			\hline
			$(\int_{0}^{T}u_{\theta}) $ & 6.49e6&6.74e6&7.02e6&7.33e6&7.65e6&7.98e6&8.32e6&8.67e6&9.02e6&9.38e6&9.74e6&1.01e7&1.04e7\\
			\hline
		\end{tabular}
					\caption{Intervention time and control cost for different values of $\theta$}
			\label{tab:backst}
	\end{table}
\begin{remark}

For $\lambda= $10 and for $\theta = $170 we
 obtain nearly the same control cost for the two
  different interventions, but the convergence
  time for the $u_{\theta}$ intervention is smaller.
  This means that for the same control cost, the $u_{\theta}$
  intervention saves time.

  For $(\lambda,\theta) = (13 ,150)$ and $(17, 210)$, the two
  interventions give approximately the same convergence time,
  but the cost is least for the $u_{\theta}$
  intervention. We conclude that for the same
  convergence times, the $u_{\theta}$ intervention
  offers a better cost.

In conclusion, we note that thanks to
the $\alpha$ and $\theta$ control parameters,
despite the non-linearity of the  backstepping control,
 it offers a better result in terms of both
 convergence time and control cost.

\end{remark}

\section{Conclusion}
We have built feedback laws that stabilize the SIT dynamical model and have studied their robustness with respect to changes of  parameters. We study three types of feedback laws:
\begin{itemize}
\item[1)] a backstepping one in section \ref{backsteppingSubsection}.
\item[2)] one depending linearly on the total number of male mosquitoes, $M+M_s$ in section \ref{sec:male-mosquitoes}.
\item[3)] one depending linearly on the number of wild male mosquitoes $M$ in section \ref{sec:fertile-male-mosquitoes}.
\end{itemize}

For the first one we were able to prove the global asymptotic stability. Based on the analysis done in  section \ref{see:Comparativesect} we see that this  feedback law gives a better result in terms of both convergence time and control cost. However, it depends on three variables $(E,M$ and $M_s)$ which may be difficult to measure in the field.

For the second one, we proved the global asymptotic stability only in a certain invariant set $\mathcal{M}$. We conjecture that this feedback gives global stability and we show numerical evidence for this conjecture (see figure \ref{conjecturedynamics}). The advantage of this feedback law is that it depends only on the total number of male mosquitoes $M+M_s$ which is a natural quantity to measure in the field.

However, this feedback law has an important drawback due to the narrow interval allowed for the gain $\alpha$ of the feedback in \eqref{eq:hypotheses-alpha2}. This might pose a problem for the robustness of this method relative to the variations of the biological parameters.

For the third one, we proved the global asymptotic stability only in a certain invariant set $\mathcal{M}$. We also conjecture that this feedback gives global stability and we show numerical evidence for this conjecture (see figure \ref{eq:conject}). The
main difference w.r.t. the previous feedback law is that now the method is robust w.r.t. variations of the biological parameters. However, the drawback in this case is that
$M$ should be harder to measure in the field.

Changes of the environment in time, and in particular seasons (both in tropical and in temperate climates), are known to have a big impact on the mosquito populations and it will thus be important to take them into account in our future work.

Also in our work, we did not consider the pest population's spatial distribution. This has again an impact in practical terms and has been considered in several mathematical works and, in particular, those concerning invasion wave blocking \cite{almeida2022blocking},  the rolling carpet strategy \cite{almeida2023rolling} or a space dependent mosquito carrying capacity \cite{almeida:hal-04196465}.  In our future works, we will construct observers  that can estimate the state from easily measurable variables  (after this paper was submitted a first observer construction was done in \cite{observerAgbo}) and we will also integrate the spatial aspect in this dynamical model. After the first version of this paper, other output feedback results using reinforcement learning were obtained in \cite{bidi2023reinforcement, bidi2023reinforcement2}.

As stated in the introduction, although the paper is mostly written for the specific case of mosquitoes, our results can be extended to the case of other pests for which the Sterile Insect Technique is pertinent.

\section*{Acknowledgements}

The authors would like to thank Herv\'{e} Bossin and Ren\'{e} Gato for the very interesting discussions that helped them identify feedback laws that can be useful for field applications and to be aware of their limitations. We hope that our future collaborations will allow us to develop and apply the ideas put forward in this work in field interventions and learn from the results to be able to improve our strategies.

\bibliographystyle{plain}
\bibliography{biblio}

\end{document}